%% file: 2025-1.tex
\documentclass{amsart}

\include{macro}

\usepackage[dvipsnames]{xcolor}

\newif\ifcomments
\commentstrue   

\commentsfalse  

\title{From Polynomials to Databases: Arithmetic Structures in Galois Theory}
\author{Jurgen Mezinaj}
\address{Department of Mathematics and Statistics, \\
Oakland University, \\
Rochester, MI, 48309}



\begin{document}

\maketitle

\begin{abstract}
We develop a computational framework for classifying Galois groups of irreducible degree-7 polynomials over~$\Q$, combining explicit resolvent methods with machine learning techniques. A database of over one million normalized projective septics is constructed, each annotated with algebraic invariants~$J_0, \dots, J_4$ derived from binary transvections. For each polynomial, we compute resolvent factorizations to determine its Galois group among the seven transitive subgroups of~$S_7$ identified by Foulkes. Using this dataset, we train a neurosymbolic classifier that integrates invariant-theoretic features with supervised learning, yielding improved accuracy in detecting rare solvable groups compared to coefficient-based models. The resulting database provides a reproducible resource for constructive Galois theory and supports empirical investigations into group distribution under height constraints. The methodology extends to higher-degree cases and illustrates the utility of hybrid symbolic-numeric techniques in computational algebra.
\end{abstract}


\section{Introduction}\label{sec:intro}

Galois theory provides a foundational framework for understanding the solvability of polynomial equations in terms of field extensions and their automorphism groups. Originating in the early 19th century, the theory explains why radical solutions exist for quadratic, cubic, and quartic equations, but not in general for quintics or higher degrees, due to the non-solvability of the symmetric group \(S_n\) for \(n \geq 5\).

Subsequent developments by Lagrange, Jordan, Klein, and others introduced computational tools such as resolvent polynomials to study the structure of Galois groups. These tools have become central to both theoretical investigations and practical applications in number theory and algebraic geometry. A longstanding open problem in the area, the \emph{Inverse Galois Problem}, asks whether every finite group arises as the Galois group of a field extension over \(\Q\). While the problem remains unresolved in general, progress has been made for several classes, including abelian and many solvable groups. The problem is closely related to Hilbert's 12th problem, which seeks explicit descriptions of abelian extensions via transcendental methods. Although certain abelian cases have been addressed through class field theory and complex multiplication, non-abelian generalizations are still not well understood.

This work investigates the Galois groups of irreducible septic polynomials over \(\Q\), integrating classical algebraic techniques with modern computational tools. In particular, we construct a large-scale database of over 1.18 million irreducible septics (filtered by height \(h \leq 4\)), enriched with algebraic invariants derived from transvections of binary forms. Each polynomial is analyzed through explicit resolvent constructions to identify its Galois group among the seven transitive subgroups of \(S_7\): \(S_7\), \(A_7\), \(\PSL(3,2)\), \(C_7 \rtimes C_6\), \(C_7 \rtimes C_3\), \(D_7\), and \(C_7\).

Our method combines symbolic invariant computations~\cite{Sha-1} with supervised learning via neurosymbolic networks~\cite{Sha-2}, trained to classify Galois groups using both raw coefficients and invariant-based features. The use of projective normalization and height bounds follows the approach of~\cite{Sha-3}, enabling scalable enumeration while ensuring algebraic diversity. Inspired by connections to monodromy groups of superelliptic curves and their Igusa invariants~\cite{book}, we explore how invariant-theoretic features correspond to root symmetries and impact classification accuracy.

Experimental results indicate that symbolic invariants enhance the identification of rare Galois groups (e.g., solvable groups such as \(C_7 \rtimes C_6\)), particularly when applied to datasets where the dominant class \(S_7\) is removed. The inclusion of invariants significantly improves classification performance in imbalanced settings. These findings support the viability of hybrid symbolic-ML pipelines for automating group identification in computational Galois theory.

This study builds on and extends previous work in invariant-based arithmetic statistics~\cite{Sha-1}, hybrid symbolic/ML models~\cite{Sha-2}, and moduli space enumeration~\cite{Sha-3}. By explicitly constructing resolvent polynomials tailored for distinguishing subgroups of \(S_7\), following techniques from~\cite{book}, we enable large-scale validation of classical methods and uncover empirical patterns in Galois group distributions.

Recent research has demonstrated the potential of machine learning in number-theoretic contexts, including graded neural architectures~\cite{Shaska-graded}, rational function classification~\cite{2024-04}, and cryptographic moduli spaces~\cite{Sha-3}. Our contribution fits within this broader direction, applying computational techniques to the explicit realization of Galois groups and the analysis of their statistical properties.

The structure of the paper is as follows. Section~\ref{sec:prelims} introduces background on Galois extensions and invariants. Section~\ref{sec:resolvents} reviews resolvent polynomials and their factorization for septics. Section~\ref{sec:database} describes the construction and augmentation of the septic polynomial database. Section~\ref{sec:ml-septics} presents machine learning classification results. Section~\ref{sec:background} discusses the realization of cyclic and solvable groups through constructive methods. Finally, Section~\ref{sec:conclusions} outlines conclusions and future directions.

The overarching aim is to facilitate scalable and interpretable approaches to classifying Galois groups, with potential applications in arithmetic statistics, inverse Galois realizations, and computational algebra.

\section{Preliminaries on Galois theory and septics polynomials}\label{sec:prelims}

Let \(\F\) be a perfect field with characteristic \(\operatorname{char}(\F) = 0\). Consider a polynomial \(f(x) \in \F[x]\). We say that \(f(x)\) is \emph{irreducible} over \(\F\) if it cannot be expressed as the product of two non-constant polynomials with coefficients in \(\F\); that is, if there do not exist non-constant polynomials \(g(x), h(x) \in \F[x]\) such that \(f(x) = g(x)h(x)\). For any polynomial \(f(x) \in \F[x]\), there exists an extension field \(K\) of \(\F\) containing at least one root \(\alpha\) of \(f(x)\), and the smallest such extension \(E_f\) in which \(f(x)\) factors completely into linear factors is called its \emph{splitting field} over \(\F\).

\subsection{Galois groups of irreducible polynomials}
A polynomial \(f(x) \in \F[x]\) of degree \(n\) is \emph{separable} if it has \(n\) distinct roots in its splitting field, or equivalently, if it factors as \(f(x) = (x - \alpha_1) \cdots (x - \alpha_n)\) with distinct \(\alpha_i\) in \(E_f\). An extension \(E\) of \(\F\) is a \emph{separable extension} if every element of \(E\) is a root of some separable polynomial in \(\F[x]\). Given a finite field extension \(K / \F\), we say that \(K\) is \emph{Galois} over \(\F\), or that \(K / \F\) is a Galois extension, if
\[
|\text{Aut}(K / \F)| = [K : \F],
\]
where \(\text{Aut}(K / \F)\) denotes the group of automorphisms of \(K\) fixing \(\F\), and \([K : \F]\) is the degree of the extension. In this case, \(\text{Aut}(K / \F)\) is called the \emph{Galois group} of \(K / \F\), denoted \(Gal(K / \F)\).\\
A field extension \(E/F\) is called \emph{normal} if every irreducible polynomial in \(F[x]\) with a root in \(E\) splits completely over \(E\).
A finite extension \(E/F\) is called a \emph{Galois extension} if it is both normal and separable.
\begin{thm}[Fundamental Theorem of Galois Theory]
Let \(E/F\) be a finite Galois extension with Galois group \(G = \Gal(E/F)\). Then:
\begin{enumerate}
    \item There is a bijection between the set of intermediate fields \(F \subseteq K \subseteq E\) and the set of subgroups \(H \leq G\), given by \(K \mapsto \Gal(E/K)\) and \(H \mapsto E^H\).
    \item For any intermediate field \(K\), we have
    \[
    [E:K] = |\Gal(E/K)|, \qquad [K:F] = [G:\Gal(E/K)].
    \]
    \item The extension \(K/F\) is normal if and only if \(\Gal(E/K)\) is a normal subgroup of \(G\). In this case,
    \[
    \Gal(K/F) \cong G / \Gal(E/K).
    \]
\end{enumerate}
\end{thm}

\begin{proof}
The proof follows from the properties of fixed fields and the fact that for a Galois extension, the Galois group acts faithfully. The bijection is anti-inclusion: larger subgroups fix smaller fields. The degree relations come from Lagrange's theorem in the group. Normality of the subgroup corresponds to normality of the extension via the quotient group isomorphism. For a full proof, see \cite[Chapter 1]{book}.
\end{proof}

Let \(L/F\) be a finite field extension. For \(\theta \in L\), the \emph{norm} and \emph{trace} of \(\theta\) over \(F\) are defined as
\[
N_{L/F}(\theta) = \prod_{\sigma \in \Gal(L/F)} \sigma(\theta), \qquad
T_{L/F}(\theta) = \sum_{\sigma \in \Gal(L/F)} \sigma(\theta),
\]
when \(L/F\) is Galois.
A Galois extension \(L/F\) is called \emph{cyclic} if its Galois group \(\Gal(L/F)\) is a cyclic group.

\begin{thm}
Let \(F\) be a field containing a primitive \(n\)-th root of unity. Assume \(\operatorname{char}(F)=p\) with \((n,p)=1\). Then the following are equivalent:
\begin{enumerate}
    \item \(L/F\) is a cyclic extension of degree \(d \mid n\).
    \item \(L = F(\theta)\) where \(\theta\) satisfies \(\theta^d = a \in F\).
    \item \(L\) is the splitting field of \(x^d - a\) over \(F\), for some \(a \in F\).
\end{enumerate}
\end{thm}

\begin{proof}
The equivalence is from Kummer theory: the cyclic extension is generated by a d-th root, and the splitting field of the Kummer equation \(x^d - a = 0\) has Galois group cyclic of order d when the base field contains the d-th roots of unity. For details, see \cite[Chapter 1]{book}.
\end{proof}

Let \(n \in \mathbb{Z}_{>0}\). The \(n\)-th \emph{cyclotomic extension} of \(F\) is the splitting field of \(x^n-1\) over \(F\), denoted \(F(\zeta_n)\), where \(\zeta_n\) is a primitive \(n\)-th root of unity.
Now, consider an irreducible and separable polynomial \(f(x) \in \F[x]\) of degree \(n\), factoring in its splitting field \(E_f\) as
\[
f(x) = (x - \alpha_1) \cdots (x - \alpha_n),
\]
where \(\alpha_1, \ldots, \alpha_n\) are distinct roots. Since \(E_f\) is both normal (as it is a splitting field) and separable over \(\F\), it is a Galois extension. The Galois group of the polynomial, \(Gal_{\F}(f) = Gal(E_f / \F)\), consists of automorphisms permuting the roots \(\alpha_1, \ldots, \alpha_n\). For any distinct roots \(\alpha_i\) and \(\alpha_j\), there exists \(\sigma \in Gal_{\F}(f)\) such that \(\sigma(\alpha_i) = \alpha_j\), implying that \(Gal_{\F}(f)\) acts transitively on the roots. Thus, \(Gal_{\F}(f)\) embeds naturally into the symmetric group \(S_n\), the group of all permutations on \(\{1, 2, \ldots, n\}\), with this embedding defined uniquely up to conjugacy due to the arbitrary ordering of the roots.

\subsection{Septic polynomials}\label{sec:septic polynomials}

Let \(f(x)\in \Q[x]$ be a monic irreducible septic polynomial given by
\begin{equation}\label{eq:f-septic}
 f(x) = x^7 + a_6 x^6 + \cdots + a_0  
 \end{equation}
 with roots \(\alpha_1, \dots, \alpha_7\) in its splitting field \(E_f\) over \(\Q\). 
 The Galois group \(G = \Gal(E_f / \Q)\) is a transitive subgroup of \(S_7\) (up to conjugacy). 
Using GAP, we can compute all transitive subgroups of \( S_n \) for a given \( n \).\\
For our case degree 7, these are seven transitive subgroups of $S_7$: 

\begin{table}[htbp]
\caption{Transitive Subgroups of \(S_7\) and Their Orders}
\label{tab:transitive_subgroups_S7}
\centering
\begin{tabular}{|c|c|c|}
\hline
 & \textbf{Subgroup} & \textbf{Order} \\
\hline
1 & \(C_7\) & \(7\) \\
2 & \(D_7\) & \(14\) \\
3 & \(C_7 \rtimes C_3\) & \(21\) \\
4 & \( C_7 \rtimes C_6\) & \(42\) \\
5 & \(L(3,2)\) & \(168\) \\
6 & \(A_7\) & \(2520\) \\
7 & \(S_7\) & \(5040\) \\
\hline
\end{tabular}
\end{table}

\begin{figure}[htbp] 
   \centering
   \includegraphics[width=2in]{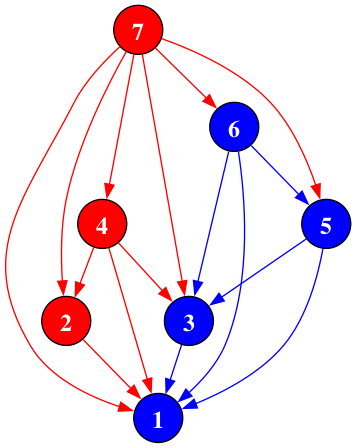} 
   \caption{The lattice of subgroups of $S_7$}
   \label{fig:S7}
\end{figure}
The relationships among these transitive subgroups form the following lattice (Figure \ref{fig:S7}), with key inclusion chains as follows:
            \[
            C_7 \subset D_7 \subset C_7 \rtimes C_3
            \]
            \[
             \quad C_7 \subset C_7 \rtimes C_3 \subset L(3,2)   \subset A_7,
            \]
            \[
             \quad C_7 \rtimes C_3 \subset C_7 \rtimes C_6
            \]
\subsection{Invariants of septics}

The action of \(Gal_{\F}(f)\) on the roots suggests a natural connection to invariant theory, where we seek polynomials in the roots that remain unchanged under this group action.
Finding the generators of the ring of invariants \(\mathcal{R}_d\) is a classical problem tackled by many XIX-century mathematicians. Such invariants are generated in terms of transvections or root differences. For binary forms \(f, g \in \mathcal{V}_d\), the \(r\)-th transvection \((f, g)_r\) is defined as a differential operator applied to the forms as detailed in \cite{curri2021}.

For degree 7, a generating set of \(\mathcal{R}_7\) is given by \(\xi = [\xi_0, \xi_1, \xi_2, \xi_3, \xi_4]\) with weights \(\mathbf{w} = (4, 8, 12, 12, 20\). We define them as follows. Let
\[
c_1 = (f, f)_6, \quad c_2 = (f, f)_4, \quad c_4 = (f, c_1)_2, \quad c_5 = (c_2, c_2)_4, \quad c_7 = (c_4, c_4)_4
\]
and
\[
\begin{aligned}
& \xi_0 = (c_1, c_1)_2, \quad \xi_1 = (c_7, c_1)_2, \quad \xi_2 = ((c_5, c_5)_2, c_5)_4, \\
& \xi_3 = ((c_4, c_4)_2, c_1^3)_6, \quad \xi_4 = ([(c_2, c_5)_4]^2, (c_5, c_5)_2)_4
\end{aligned}
\]
These invariants form a basis for \(\mathcal{R}_7\), capturing the \(\mathrm{SL}_2(\mathbb{Q})\)-invariant properties of binary septic forms. They are used as features in our neurosymbolic network to classify Galois groups.

A fundamental invariant for a monic septic polynomial \(f(x)  \in \mathbb{Q}[x]\) with roots \(\alpha_1, \ldots, \alpha_7\) in its splitting field \(E_f\) is the \emph{discriminant}, defined as
\[
\Delta_f = \prod_{1 \leq i < j \leq 7} (\alpha_i - \alpha_j)^2.
\]
Equivalently, it can be computed via the resultant:
\[
\Delta_f = (-1)^{21} \res(f, f') = -\res(f, f'),
\]
where \(f'(x)\) is the formal derivative of \(f\). The resultant is the determinant of the \(13 \times 13\) Sylvester matrix constructed from the coefficients of \(f\) and \(f'\), making \(\Delta_f\) symbolically computable for degree 7. Moreover, \(\Delta_f \neq 0\) if and only if \(f\) has distinct roots, and \(\Delta_f > 0\) if and only if the Galois group is contained in \(A_7\)

\section{Resolvents}\label{sec:resolvents}
To aid in computing \(Gal(f)\), we introduce the \emph{resolvent polynomial} for a degree $n$ polynomial and then apply the results for $n=7$. 

Given a function \(F(x_1, \ldots, x_n) \in K[x_1, \ldots, x_n]\), often a polynomial symmetric under some subgroup of \(S_n\), we define the resolvent polynomial associated with a polynomial \(f\), a subgroup \(G \subseteq S_n\), and \(F\) as follows:

\begin{defn}
The resolvent polynomial \( \Res{G}{f}{F} (x) \) is given by
\[
\Res{G}{f}{F} (x) = \prod_{\sigma \in G / H} (x - \theta_\sigma),
\]
where \(\theta_\sigma = F(r_{\sigma(1)}, \ldots, r_{\sigma(n)})\), the roots \(r_1, \ldots, r_n\) (sometimes denoted \(\alpha_1, \ldots, \alpha_n\)) are those of \(f(x)\) in its splitting field, and 
\[
H = \{ \tau \in G \mid F(r_{\tau(1)}, \ldots, r_{\tau(n)}) = F(r_1, \ldots, r_n) \}
\]
 is the stabilizer of \(F\) under the action of \(G\). The product is taken over coset representatives of \(G / H\), ensuring each distinct value \(\theta_\sigma\) appears exactly once, and the degree of the resolvent is \(k = |G| / |H|\), the index of \(H\) in \(G\).
\end{defn}

The roots \(\theta_\sigma\) form the orbit of \(\theta_e = F(r_1, \ldots, r_n)\) under \(G\), with their distinctness depending on the symmetry of \(F\). For example, if \(F = x_1\) and \(G = S_n\), then \(H = S_{n-1}\) (fixing the first index), \(k = n\), and \(\Res{S_n}{f}{x_1} = f(x)\); if \(F\) is fully symmetric, \(H = G\) and \(k = 1\). To compute \(\Res{G}{f}{F} (x)\) symbolically, we express it as
\[
\Res{G}{f}{F} (x) = x^k - e_1 x^{k-1} + e_2 x^{k-2} - \cdots + (-1)^k e_k,
\]
where \(e_j\) are the elementary symmetric polynomials in the \(k\) distinct \(\theta_\sigma\). Since the roots \(r_i\) are typically not known explicitly, we rely on the elementary symmetric sums of \(f(x) = x^n + a_{n-1}x^{n-1} + \cdots + a_1x + a_0\):
\[
\begin{split}
s_1 &	= r_1 + \cdots + r_n = -a_{n-1}, \\
s_2 &	= \sum_{i < j} r_i r_j = a_{n-2}, \\
	& \vdots \\
s_n &	= r_1 \cdots r_n = (-1)^n a_0
\end{split}
\]
and compute the power sums \(p_m = \sum_{\sigma \in G / H} \theta_\sigma^m\) in terms of the \(s_i\), using Newton’s identities to derive the \(e_j\). Newton's identities are recursive relations that allow computation of the elementary symmetric polynomials from power sums or vice versa. Specifically, the identities are given by
\[
p_k - e_1 p_{k-1} + e_2 p_{k-2} - \cdots + (-1)^{k-1} e_{k-1} p_1 + (-1)^k k e_k = 0
\]
for \(k = 1, \dots, n\), solving for \(e_k\) sequentially. This construction, rooted in computational algebraic number theory (e.g., \cite{cohen}), facilitates tasks such as determining Galois groups or factoring polynomials over extensions. Once the transitive subgroups of \(S_n\) are classified, identifying which corresponds to \(Gal(f)\) becomes the next step, detailed in the following sections.

\begin{thm}Let $m=[G: H]=deg\left(R_G(F, f)\right)$. Then, if $R_G(F, f)$ is squarefree, its Galois group (as a subgroup of $S_m$) is equal to $\phi(Gal(f))$, where $\phi$ is the natural group homomorphism from $G$ to $S_m$ given by the natural left action of $G$ on $G / H$.
\end{thm}

In particular, the list of degrees of the irreducible factors of $R_G(F, f)$ in $\mathbb{Z}[x]$ is the same as the list of the lengths of the orbits of the action of $\phi(\Gal(f))$ on $\{1, \ldots, m\}$. For example, $R_G(F, f)$ has a root in $\mathbb{Z}$ if and only if $Gal(f)$ is conjugate under $G$ to a subgroup of $H$.
For the proof, see $[So1]$.

Consider a generic irreducible polynomial
\begin{equation}\label{eq:f}
 f(x) = \prod_{i=1}^n (x-x_i) = x^n + a_{n-1}x^{n-1} + \cdots + a_0 \in \Q[x]
\end{equation}
 where its roots $x_1, \ldots , x_n$ are considered variables.
Then $S_n$ acts on $\Q[x_1, \ldots, x_n] $ by permuting the variables.
 \begin{equation}
 \begin{split}
 S_n \times \Q[x_1, \ldots, x_n] & \to \Q[x_1, \ldots, x_n] \\
\left( \tau, F(x_1, \ldots, x_n) \right) & \to F( \tau(x_1) , \ldots, \tau(x_n) =: F^\tau
\end{split}
\end{equation}
For any \( G \subseteq S_n \) a polynomial $F(x_1, \ldots , x_n)$ is called \emph{symmetric under \( G \)} if $F=F^\tau$ for all $\tau \in G$.
Let $H$ denote the stabilizer of $F$ in $G$
\[
 H = \{ \tau \in G \mid F=F^\tau \} .
 \]
The \textbf{resolvent polynomial of $f(x)$ with respect to $F$} , denoted by \( \Res{G}{f}{F} \), is defined as
\[
\Res{G}{f}{F} = \prod_{\sigma \in G / H} \left( x -F^\sigma (x_1, \ldots , x_n) \right).
\]
 The product is over coset representatives of \( G / H \), and the degree of the resolvent is \( k = |G| / |H| \).
The resolvent’s factorization over \( \Q \) reveals information about the Galois group \( \Gal(f) \), as its irreducible factors correspond to the orbits of 
\( \Gal(f) \) acting on \( G / H \).

\begin{thm}
The factorization of \(\Res{G}{f}{F}\) over \(\Q\) into irreducible factors has degrees equal to the orbit lengths under the left action of \(\Gal(f)\) on the right cosets \(G / H\). The number of factors is the number of double cosets \(\Gal(f) \backslash G / H\).
\end{thm}

Let us now go through the steps of how one would compute the resolvent when given $G$, $F$ and a generic irreducible polynomial $f(x)$ given in 
Let $H$ be the stabilizer of \( F \). Then $\leq G$. Denote by \( k = |G| / |H| \) the index of $H $ in $G$.\\
%
 
\subsubsection{Express Roots Symbolically}:
 The roots \( \theta_\sigma = F(x_{\sigma(1)}, \ldots, x_{\sigma(n)}) \) are functions of the roots \( x_i \). Use Vieta’s formulas for the elementary symmetric sums of \( f(x) \):
\[
\begin{split}
    s_1 & = x_1 + \cdots + x_n = -a_{n-1}, \\
    s_2 & = \sum_{i < j} x_i x_j = a_{n-2}, \\
     \vdots & \\
    s_n & = x_1 \cdots x_n = (-1)^n a_0.
\end{split}
\]

\subsubsection{Compute Power Sums}: Define the power sums
    \[
    p_m = \sum_{\sigma \in G / H} \theta_\sigma^m = \sum_{\sigma} [F(x_{\sigma(1)}, \ldots, x_{\sigma(n)})]^m.
    \]
    Expand \( F(x_{\sigma(1)}, \ldots, x_{\sigma(n)})^m \), sum over coset representatives, and express the result in terms of \( s_1, \ldots, s_n \) using symmetric polynomial identities. For example, if \(F\) is a linear form, \(p_m\) can be expressed using binomial expansions and symmetric sums; for higher-degree \(F\), use multivariate generating functions or character-theoretic projections to compute the sums efficiently.

\subsubsection{Apply Newton’s Identities}: Relate \( p_m \) to \( e_j \) via Newton’s identities:
    \[
    \begin{split}
    e_1 & = p_1,\\
    e_2 & = \frac 1 2 (e_1 p_1 - p_2), \\\
    e_3 & = \frac 1 3 \left(e_2 p_1 - e_1 p_2 + p_3\right), \\
       \vdots & \\
    e_j & =\frac 1 j \left( \sum_{i=1}^{j-1} (-1)^{i-1} e_{j-i} p_i + (-1)^{j-1} p_j \right).
    \end{split}
    \]
Solve recursively to obtain \( e_1, \ldots, e_k \). This recursion is stable for symbolic computation, allowing exact rational coefficients when starting from rational \(a_i\).

\subsubsection{Construct the Resolvent}: Form the polynomial using the computed \( e_j \).
With \( e_1, e_2, \ldots, e_k \) computed, the resolvent is
\[
\Res{G}{f}{F} = x^k - e_1 x^{k-1} + e_2 x^{k-2} - \cdots + (-1)^k e_k.
\]
This polynomial has degree \( k \), and its coefficients are fully symbolic in the coefficients of \( f(x) \). Symbolic computation is exact but computationally intensive for high-degree resolvents, especially for septic polynomials with large \( k \). For instance, the 120-ic resolvent for septics has coefficients of enormous degree in the \(a_i\), requiring computer algebra systems like Sage or Magma for practical evaluation.

\section{Resolvents of septics}\label{sec:galois_septics}

Consider the irreducible polynomial $f \in \mathbb{Q}[x]$ in \eqref{eq:f-septic}. 
Let its roots $\alpha_1, \alpha_2, \dots, \alpha_7$ lie in its splitting field.
 The goal in Galois theory is to determine the Galois group \(\Gal(f)\subseteq S_7\), as this group provides information about the solvability of the polynomial and the structure of its splitting field. To achieve this characterization, we rely on the construction and factorization of \textit{resolvent polynomials}, which are specialized polynomials whose roots are certain algebraic expressions involving the original roots \(\alpha_i\).

Resolvent polynomials are formed by considering subgroups \(H\subseteq S_7\) and functions \(F(\alpha_1,\dots,\alpha_7)\) invariant under the subgroup \(H\). The resolvent associated with \(H\) is then constructed by taking the product of conjugates of this invariant under the cosets of \(H\) in \(S_7\):
\[
R_H(x) = \prod_{\sigma \in S_7/H}(x - F^\sigma),
\]
where \(F^\sigma\) denotes the action of permutation \(\sigma\) on \(F\). The degree of \(R_H(x)\) is given by the index \([S_7:H]\), which guides us towards a deeper understanding of the Galois group structure by analyzing its factorization pattern over the rational field.

The simplest resolvent polynomial for septic equations, known as the quadratic resolvent, employs the discriminant \(\Delta = \prod_{i<j}(\alpha_i - \alpha_j)^2\). This invariant reflects whether permutations in the Galois group are even or odd. The quadratic resolvent is then expressed as:
\[
R_1(x) = x^2 - \Delta.
\]
Its factorization pattern directly tests the inclusion of \(\Gal(f)\) within the alternating group \(A_7\). Explicitly, if the quadratic resolvent splits into linear factors over \(\Q\), the Galois group must be a subgroup of \(A_7\). Conversely, irreducibility of \(R_1(x)\) indicates the presence of odd permutations, confirming that \(\Gal(f)\) is exactly \(S_7\).

To distinguish more intricate subgroups, we examine the 30-ic resolvent constructed from an invariant associated with the subgroup \(\PSL(3,2)\cong L(3,2)\). This invariant is explicitly given by the polynomial expression:
\[
F_2 = \alpha_3\alpha_1\alpha_4 + \alpha_4\alpha_2\alpha_5 + \alpha_5\alpha_3\alpha_6 + \alpha_6\alpha_4\alpha_7 + \alpha_7\alpha_5\alpha_1 + \alpha_1\alpha_6\alpha_2 + \alpha_2\alpha_7\alpha_3,
\]
which remains fixed under the action of \(\PSL(3,2)\). The 30-ic resolvent polynomial is then:
\[
R_2(x) = \prod_{\sigma \in S_7/\PSL(3,2)}(x - F_2^\sigma).
\]
This resolvent discriminates the simple subgroup \(L(3,2)\) by factorization: complete irreducibility corresponds to the full symmetric group, whereas certain prescribed factorizations such as degrees \(1,7,8,14\) uniquely signal \(\Gal(f)\subseteq L(3,2)\).

For deeper subgroup identification, particularly for the metacyclic group \(F_{42} = C_7\rtimes C_6\), the 120-ic resolvent plays an essential role. Defined through a carefully chosen invariant polynomial:
\[
\begin{aligned}
F_3 = &\alpha_3\alpha_1(\alpha_4 + \alpha_7) + \alpha_2\alpha_5(\alpha_4 + \alpha_3) + \alpha_5\alpha_6(\alpha_3 + \alpha_7) \\
&+\alpha_4\alpha_6(\alpha_7 + \alpha_3) + \alpha_5\alpha_1(\alpha_7 + \alpha_6) + \alpha_1\alpha_2(\alpha_6 + \alpha_4) + \alpha_2\alpha_7(\alpha_3 + \alpha_6),
\end{aligned}
\]
the resulting 120-ic resolvent is constructed as:
\[
R_3(x) = \prod_{\sigma \in S_7/F_{42}}(x - F_3^\sigma).
\]
The explicit factorization patterns of this polynomial over \(\Q\), such as degrees \(1,7,14,21,42\), conclusively determine whether \(\Gal(f)\subseteq F_{42}\). Membership in this subgroup is particularly significant, as it implies solvability by radicals, connecting resolvent computations directly to classical algebraic solvability criteria.

Additionally, resolvent polynomials formed by summation over distinct subsets of roots provide another effective method of subgroup detection. For instance, the 35-ic resolvent polynomial defined by summation of triplets of roots:
\[
R(x) = \prod_{1\leq i<j<k\leq 7}(x - (\alpha_i+\alpha_j+\alpha_k)),
\]
with stabilizer subgroup \(S_3\times S_4\), offers precise patterns of factorization that distinctly correspond to subgroups such as the dihedral group \(D_7\), the Frobenius group \(F_{21}\), and others, greatly enriching the classification framework.

Resolvent polynomials provide a systematic way to determine \(G\) by studying the factorization patterns over \(\Q\) of auxiliary polynomials constructed from the roots \(\alpha_i\). These methods trace back to Lagrange and Jordan and were refined for septic polynomials by Foulkes. For a subgroup \(H \leq S_7\) and a polynomial \(F(x_1, \dots, x_7) \in \Q[x_1, \dots, x_7]\) with stabilizer \(\Stab_{S_7}(F) = H\), the associated resolvent is
\[
R_F(f)(y) = \prod_{\tau \in H \backslash S_7} \bigl( y - F(\alpha_{\tau(1)}, \dots, \alpha_{\tau(7)}) \bigr) \in \Q[y],
\]
a monic polynomial of degree \([S_7 : H] = 5040 / |H|\). The irreducible factor degrees of \(R_F(f)(y)\) correspond to orbit lengths under the left action of \(G\) on the right cosets \(H \backslash S_7\), and the number of irreducible factors equals the number of double cosets \(G \backslash S_7 / H\) \cite{book}.

\begin{enumerate}[label=\upshape(\roman*)]
\item If \(G = S_7\), then \(R_F(f)(y)\) is irreducible.
\item If \(G \leq H^\sigma\) for some conjugate \(H^\sigma\) of \(H\), then \(R_F(f)(y)\) splits completely into linear factors over \(\Q\).
\item For a maximal proper subgroup \(H < G' \leq S_7\) with \(G \leq G'\), if \(R_F(f)(y)\) has a rational root, then \(G \leq H\); otherwise \(G = G'\).
\end{enumerate}

Foulkes  constructed three key resolvents for septics: a quadratic resolvent for \(A_7\), a degree-30 resolvent for \(\PSL(3,2)\), and a degree-120 resolvent for \(F_{42}\). Their factorizations distinguish all transitive subgroups of \(S_7\).

\subsection{The Quadratic Resolvent}
This resolvent detects whether \(G \subseteq A_7\):

\begin{enumerate}[label=\upshape(\roman*)]
\item Let \(H = A_7\), and let \(\Delta = \prod_{i < j} (\alpha_i - \alpha_j)^2\) be the discriminant of \(f\).
\item Define \(F_1 = \sqrt{\Delta}\), which changes sign under odd permutations.
\end{enumerate}
The corresponding resolvent is
\[
R_1(x) = (x - \sqrt{\Delta})(x + \sqrt{\Delta}) = x^2 - \Delta,
\]
of degree \([S_7 : A_7] = 2\). Thus:

\begin{enumerate}[label=\upshape(\roman*)]
\item If \(G \subseteq A_7\), then \(\sqrt{\Delta} \in \Q\) and \(R_1(x)\) splits into linear factors.
\item If \(G = S_7\), then \(R_1(x)\) is irreducible.
\end{enumerate}

The discriminant \(\Delta\) is computed as \(\mathrm{Res}(f, f')\) \cite{book}.

\subsection{The 30-ic Resolvent for \(\PSL(3,2)\)}
This resolvent tests for inclusion in \(\PSL(3,2)\), a maximal subgroup of \(S_7\):

\begin{enumerate}[label=\upshape(\roman*)]
\item Let \(H = \PSL(3,2)\), and define
\[
F_2 = \alpha_3 \alpha_1 \alpha_4 + \alpha_4 \alpha_2 \alpha_5 + \alpha_5 \alpha_3 \alpha_6 + \alpha_6 \alpha_4 \alpha_7 + \alpha_7 \alpha_5 \alpha_1 + \alpha_1 \alpha_6 \alpha_2 + \alpha_2 \alpha_7 \alpha_3,
\]
which is invariant under \(H\).
\end{enumerate}
Then
\[
R_2(x) = \prod_{\sigma \in S_7 / \PSL(3,2)} (x - F_2^\sigma),
\]
of degree \([S_7 : \PSL(3,2)] = 30\). Its factorization encodes:

\begin{enumerate}[label=\upshape(\roman*)]
\item \(G = S_7\): \(R_2(x)\) irreducible;
\item \(G = A_7\): \(R_2(x)\) factors as \(15,15\);
\item \(G = \PSL(3,2)\): factors \(1,7,8,14\).
\end{enumerate}

\subsection{The 120-ic Resolvent for \(C_7 \rtimes C_6\)}
This resolvent targets the metacyclic group \( C_7 \rtimes C_6\):

\begin{enumerate}[label=\upshape(\roman*)]
\item Let \(H = F_{42}\), and define
\[
\begin{aligned}
F_3 =\ & \alpha_3 \alpha_1 (\alpha_4 + \alpha_7) + \alpha_2 \alpha_5 (\alpha_4 + \alpha_3) + \alpha_5 \alpha_6 (\alpha_3 + \alpha_7) \\
      &+ \alpha_4 \alpha_6 (\alpha_7 + \alpha_3) + \alpha_5 \alpha_1 (\alpha_7 + \alpha_6) + \alpha_1 \alpha_2 (\alpha_6 + \alpha_4) + \alpha_2 \alpha_7 (\alpha_3 + \alpha_6),
\end{aligned}
\]
which is invariant under \(H\).
\end{enumerate}
The resolvent is
\[
R_3(x) = \prod_{\sigma \in S_7 / C_7 \rtimes C_6} (x - F_3^\sigma),
\]
of degree \([S_7 : C_7 \rtimes C_6] = 120\). If \(G \subseteq C_7 \rtimes C_6\), then \(f\) is solvable by radicals, and the factor degrees include \(1,7,14,21,21,42\).

Berwick’s invariant cubic \(z^3 - \phi' z^2 + A z - \Delta = 0\) (where \(\phi'\) involves square-root differences) refines this classification under field extensions \cite{Sha-2}.


\begin{table}[htbp]
\caption{Degrees of Irreducible Factors of Resolvents for Septic Galois Groups}
\label{tab:foulkes_resolvents}
\centering
\begin{tabular}{|c|c|c|c|}
\hline
Galois Group & \(R_1(x)\) (Quadratic) & \(R_2(x)\) (30-ic) & \(R_3(x)\) (120-ic) \\
\hline
\(S_7\) & 2 & 30 & 120 \\
\(A_7\) & 1,1 & 15,15 & 120 \\
\(\PSL(3,2)\) & 1,1 & 1,7,8,14 & 8,56,56 \\
\(C_7 \rtimes C_6\) & 2 & 2,14,14 & 1,7,14,21,21,42 \\
\(C_7 \rtimes C_3\)& 1,1 & 1,7,7,7,7 & 1,7,7,7,7,21,21,21,21 \\
\(D_7\) & 2 & 2,14,14 & 1, \( 7\times7,5\times14 \) \\
\(C_7\) & 1,1 & 1,7,7,7,7 & 1, \(17\times7 \)  \\
\hline
\end{tabular}
\end{table}
\subsection{Numerical Computation}

For high-degree resolvents or complex \( G \), numerical methods are more practical, as described by Cohen \cite{cohen}. This approach approximates the roots of \( f(x) \), computes \( \theta_\sigma \), and constructs the resolvent by rounding coefficients to integers, leveraging that \( \Res{G}{f}{F} \in \mathbb{Z}[x] \) when \( f(x) \in \mathbb{Z}[x] \) and \( F \) has integer coefficients. The algorithm is:

\subsubsection{Approximate Roots}: Compute the roots \( x_1, \ldots, x_n \in \mathbb{C} \) of \( f(x) \) to high precision (e.g., 50–100 decimal places) using a root-finding algorithm like Newton-Raphson or Laguerre’s method. Newton-Raphson iterates 
\[
x_{k+1} = x_k - f(x_k)/f'(x_k),
\]
 converging quadratically for simple roots, while Laguerre's method is more robust for polynomials, using cubic convergence via 
\[
x_{k+1} = x_k - n f(x_k) / (f'(x_k) \pm \sqrt{(n-1)^2 f'(x_k)^2 - n (n-1) f(x_k) f''(x_k)}).
\]

\subsubsection{Evaluate Resolvent Roots}: For each \( \sigma \in G / H \), compute 
\[
\theta_\sigma = F(x_{\sigma(1)}, \ldots, x_{\sigma(n)}).
\]
 The set \( \{ \theta_\sigma \} \) has \( k \) distinct values, assuming no accidental coincidences in the approximations.

\subsubsection{Compute Coefficients}: Form the resolvent
    \[
    \Res{G}{f}{F} = \prod_{\sigma \in G / H} (x - \theta_\sigma) = x^k - e_1 x^{k-1} + \cdots + (-1)^k e_k.
    \]
    Calculate \( e_j \) using power sums \( p_m = \sum_{\sigma} \theta_\sigma^m \) and Newton’s identities, as in the symbolic method. Power sums are computed by summing the high-precision complex numbers, and Newton's recursion is applied numerically.

\subsubsection{Round Coefficients}: Round each \( e_j \) to the nearest integer, ensuring accuracy with high-precision root approximations. To avoid errors, use interval arithmetic or increased precision (e.g., 200 digits) to confirm the integer values lie within error bounds less than 0.5.

\subsubsection{Verify}: Check the polynomial by evaluating at points (e.g., \( x = 0, 1 \)) or recomputing its roots, comparing to the \(\theta_\sigma\), or factoring the rounded polynomial symbolically and checking consistency with group-theoretic expectations.

Numerical methods are efficient but require careful precision management to avoid rounding errors \cite{cohen}. For example, in septic cases, approximating roots to 100 digits ensures resolvent coefficients are correctly rounded for degrees up to 120, as the condition number of the companion matrix is bounded for polynomials with bounded heights \cite{Sha-3}.


The Galois group \(\Gal(f)\) is determined uniquely by the factorization patterns of the resolvents \(T(x) = x^2 - \Delta\), \(\Psi(x) = R_2(x)\), and \(\Phi(x) = R_3(x)\).

\begin{thm}\label{thm:galois_resolvents}
The Galois group \(\Gal(f)\) is uniquely determined by the degrees of the irreducible factors of \(T(x)\), \(\Psi(x)\), and \(\Phi(x)\) as listed in \cref{tab:foulkes_resolvents}. Specifically:
\begin{enumerate}
\item \(S_7\) if \(T(x)\) and \(\Psi(x)\) are both irreducible.
\item \(A_7\) if \(T(x)\) factors as \(1,1\) and \(\Psi(x)\) as \(15,15\).
\item \(\PSL(3,2)\) if \(T(x)\) has \(1,1\) and \(\Psi(x)\) has \(1,7,8,14\).
\item \(C_7 \rtimes C_6\)if \(T(x)\) is irreducible, \(\Psi(x)\) has \(2,14,14\), and \(\Phi(x)\) has \(1,7,14,21,21,42\).
\item \(C_7 \rtimes C_3\) if \(T(x)\) has \(1,1\), \(\Psi(x)\) has \(1,7,7,7,7\), and \(\Phi(x)\) has \(1,7,7,7,7,21,21,21,21\).
\item \(C_7\) if \(T(x)\) has \(1,1\), \(\Psi(x)\) has \(1,7,7,7,7\), and \(\Phi(x)\) has \(1,17\times7\).
\item \(D_7\) if \(T(x)\) is irreducible, \(\Psi(x)\) has \(2,14,14\), and \(\Phi(x)\) has \(1,7\times7,5\times14\).
\end{enumerate}
\end{thm}

\begin{proof}
The roots of each resolvent form orbits under \(G \subseteq S_7\) acting on cosets \(S_7 / H\). The stabilizers (\(A_7\), \(\PSL(3,2)\), \(C_7 \rtimes C_6\)) ensure distinct factorization patterns, as shown in \cite{Sha-1,Sha-3}. A linear factor indicates \(G\) lies in a conjugate of the stabilizer subgroup.
\end{proof}

These factorization patterns are later used as features in the machine learning classification of septic Galois groups.

\subsection{A 35-ic Resolvent via Triplet Sums}

Soicher and McKay \cite{soicher} constructed a linear resolvent using 3-set sums of roots. Fix \(F = \alpha_1 + \alpha_2 + \alpha_3\); its stabilizer is \(H = S_3 \times S_4\) (\(|H| = 144\)), giving
\[
R(x) = \prod_{1 \leq i < j < k \leq 7} \bigl( x - (\alpha_i + \alpha_j + \alpha_k) \bigr),
\]
of degree \([S_7 : H] = 35 = \binom{7}{3}\). A Tschirnhausen transformation ensures distinct roots if needed.

To compute \(R(x)\), use power sums \(p_m = \sum_{i<j<k} (\alpha_i + \alpha_j + \alpha_k)^m\) in terms of elementary symmetric polynomials \(s_r = (-1)^r a_{7-r}\) (\(s_1 = -a_6\), etc.) via Newton identities:
\[
e_1 = p_1, \quad k e_k = \sum_{i=1}^k (-1)^{i-1} e_{k-i} p_i, \quad R(x) = x^{35} - e_1 x^{34} + \cdots + (-1)^{35} e_{35}.
\]
Explicitly,
\begin{align*}
p_1 &= 15 s_1 = -15 a_6, \\
p_2 &= 15(s_1^2 - 2 s_2) + 10 s_2 = 15 a_6^2 - 20 a_5, \\
p_3 &= -15 a_6^3 + 89 a_4, \\
&\vdots \\
p_{35} &= -35 a_0^5 + \cdots,
\end{align*}
where higher \(p_m\) follow from symbolic symmetric expansions. These can be computed symbolically or numerically for factorization.

\begin{table}[htbp]
\caption{Orbit-Length Partitions of 3-Sets under Transitive Subgroups of \(S_7\)}
\label{tab:orbit_lengths}
\centering
\begin{tabular}{|c|c|}
\hline
\(G\) & Orbit Lengths \\
\hline
\(C_7\) & \(7^5\) \\
\(D_7\) & \(7^3, 14\) \\
\(C_7 \rtimes C_3\)& \(7^2, 21\) \\
\(C_7 \rtimes C_6\)& \(21\) \\
\(\PSL(3,2)\) & \(7, 28\) \\
\(A_7\) & \(35\) \\
\(S_7\) & \(35\) \\
\hline
\end{tabular}
\end{table}


The 35-ic resolvent \(R(x)\) can determine \(G\) via factorization degrees, with auxiliary discriminant tests for ambiguous cases \cite{Sha-3}.  
For coefficients of the degree 35 resolvent see \cref{app:a}. As far as we are aware, this is the first time that this resolvent has been computed.

\begin{thm}\label{thm:3sets_resolvent}
Let \(f \in \Q[x]\) be monic, irreducible, and of degree \(7\). Then the degrees of the irreducible factors of the 35-ic resolvent \(R(x)\) determine \(\Gal(f)\) as follows:
\begin{enumerate}
\item \(7,7,7,7,7:\ G = C_7\).
\item \(7,7,7,14:\ G = D_7\).
\item \(7,7,21:\ G = C_7 \rtimes C_3\).
\item Single factor of degree \(21\):

  \begin{enumerate}[label=\upshape(\roman*)]
  \item If \(g_d(x)\) (degree \(42\), with roots \(b_k \pm \sqrt{d}\), where \(d = \mathrm{disc}(f)/\square\)) is reducible, then \(G = \PSL(3,2)\);
  \item Otherwise \(G = F_{42}\).
  \end{enumerate}
\item Irreducible (\(35\)): if \(\mathrm{disc}(f)\) is a square, \(G = A_7\); otherwise \(G = S_7\).
\end{enumerate}
\end{thm}

\begin{proof}
The factor degrees correspond to orbit sizes of 3-sets under \(G\) \cite{Sha-2}. The partitions \(7^5\), \(7^3\cdot14\), and \(7^2\cdot21\) uniquely identify \(C_7\), \(D_7\), and \(F_{21}\), respectively. For partition \(21\), reducibility of \(g_d(x)\) distinguishes \(\PSL(3,2)\) (stabilizer in \(A_7\)) from \(F_{42}\). For \(35\), the discriminant criterion distinguishes \(A_7\) from \(S_7\). Tschirnhausen transformations preserve \(G\).
\end{proof}

This single-resolvent method is computationally lighter than the triple-resolvent approach and complements it in probabilistic classification studies \cite{Sha-3}.

\section{Database of the irreducible septics}\label{sec:database}

In this section, our objective is to build a database of irreducible polynomials \( f \in \Q[x] \) of degree \( \deg f = n \). The data is organized in a Python dictionary. Each polynomial \( f(x) = \sum_{i=0}^n a_i x^i \) is represented by its corresponding binary form \( f(x, y) = \sum_{i=0}^n a_i x^i y^{n-i} \). In this way, each polynomial is identified with a point in the projective space \( \PP^n \), represented by the integer coordinates
\[
\mathbf{p} = [a_n : \cdots : a_0] \in \PP^n,
\]
where \( \gcd(a_0, \ldots, a_n) = 1 \).

Since \( f(x) \) is irreducible over \( \Q \) and has degree \( n \), we must have \( a_n \neq 0 \) and \( a_0 \neq 0 \). Moreover, its discriminant \( \Delta_f \) is nonzero.

Next, we generate a dataset of these polynomials with a bounded height \( h \). Let denote by \( \mathcal{P}_h^n \) the set of points corresponding to these polynomials, i.e.,
\[
\mathcal{P}_h^n := \{ \mathbf{p} = [a_n : \cdots : a_0] \in \PP^n \mid a_0 a_n \neq 0, \Delta_f \neq 0 \}.
\]
To guarantee that each entry in the database is unique, we index the Python dictionary by the tuple \( (a_0, \ldots, a_n) \). This approach ensures that polynomials are not recorded more than once in the Python dictionary.

For fixed \( h \) and \( n \), the cardinality of \( \mathcal{P}_h^n \) is bounded by
\[
|\mathcal{P}_h^n| \leq 4 h^2 (2 h + 1)^{n-2}.
\]
To make the bound precise, we record the following counting formula for primitive projective points of exact height $h$.
\begin{lem}[Rational Points in Projective Space]
Let $n \geq 1$ and $h \geq 1$ be integers.
Let $P_n(h)$ denote the number of primitive integer points $\mathbf{x} = (x_0, x_1, \ldots, x_n) \in \mathbb{Z}^{n+1}$, modulo sign, with $\max_i |x_i| = h$ and $\gcd(x_0, x_1, \ldots, x_n) = 1$.  
Then
\[
P_n(h) = \frac{ (2h + 1 )^{ n+1 } - ( 2h - 1 )^{ n+1 } }{2}
- \sum_{\substack{\, d \mid h \\ d \geq 2 }} P_n\left( \frac{h}{d} \right)
\]
where the sum runs over all integers $d \geq 2$ dividing $h$.
\end{lem}

For the case of degree \( d \geq 7 \) and a given height \( h \), one can construct these sets using SageMath as illustrated below:
\begin{verbatim}
PP = ProjectiveSpace(d, QQ)
rational_points = PP.rational_points(h)
\end{verbatim}

\begin{table}[h]
\caption{Counts by height: rational vs.\ irreducible}
\label{tab:height_rational_irreducible}
\centering
\begin{tabular}{|c|r|r|}
\hline
Height $h$ & Rational & Irreducible \\
\hline
1 & 3{,}280 & 916 \\
2 & 94{,}376 & 46{,}552 \\
3 & 863{,}144 & 538{,}170 \\
4 & 4{,}420{,}040 & 3{,}103{,}800 \\
\hline
\end{tabular}
\end{table}

\begin{table}[h]
\caption{Number of points in $\PP^7(\Q)$ by height}
\label{tab:points_height}
\begin{tabular}{|c|r|r|r|}
\hline
Height $h$ & Points with height $\leq h$ & Points with exact height $h$  \\
\hline
1  & 3,280         & 3,280         \\
2  & 192,032       & 188,752        \\
3  & 2,875,840     & 2,683,808     \\
4  & 21,324,768    & 18,448,928     \\
5  & 106,977,568   & 85,652,800     \\
6  & 404,787,648   & 297,810,080    \\
7  & 1,278,364,320 & 873,576,672   \\
8  & 3,466,156,768 & 2,187,792,448 \\
9  & 8,467,372,480 & 5,001,215,712 \\
10 & 18,801,175,808 & 10,333,803,328 \\
\hline
\end{tabular}
\end{table}

After generating the points, the data is normalized by clearing denominators so that all coordinates become integers. We then retain only those polynomials that are irreducible over \( \Q \). For each point \( \mathbf{p} \in \PP^n \), we compute the following:
\[
(a_0, \ldots, a_n) : [H(f), [\xi_0, \ldots, \xi_n, \Delta_f], \quad \text{sig}, \Gal_Q(f)].
\]
Here, \( H(f) \) denotes the height of \( f(x) \), \( [\xi_0, \ldots, \xi_n] \) are the generators of the ring of invariants for binary forms of degree \( n \), the discriminant \( \Delta_f \), sig is the signature, and \( \Gal_Q(f) \) indicates the GAP identifier of the Galois group.

\subsection{Irreducible Septics}
We create a database of all rational points \( \mathbf{p} \in \PP^7 \) with projective height \( h \leq 4 \) such that
\[
f(x) = a_7 x^7 + a_6 x^6 + a_5 x^5 + a_4 x^4 + a_3 x^3 + a_2 x^2 + a_1 x + a_0
\]
is irreducible over \( \Q \). After normalizing the leading
coefficient to $1$, this polynomial has the form \eqref{eq:f-septic}.
The Galois group of an irreducible septic is one of the seven transitive subgroups of \( S_7 \), as classified by Foulkes \cite{foulkes}: \( S_7 \), \( A_7 \), \( \PSL(3,2) \), \(C_7 \rtimes C_6\), \(C_7 \rtimes C_3\), \( D_7 \), or \( C_7 \). Table \ref{tab:group_counts} shows the counts of polynomials with each Galois group for height \( h \leq 4 \).

\begin{table}[h]
\caption{Counts for Groups with Height \( \leq 4 \)}
\label{tab:group_counts}
\begin{tabular}{|c|c|}
\hline
Galois Group & Count \\
\hline
\( S_7 \) & 584,324 \\
\( A_7 \) & 138 \\
\( \PSL(3,2) \) & 136 \\
\( D_7 \) & 18 \\
\(C_7 \rtimes C_6\) & 4 \\
\(C_7 \rtimes C_3\) & 0 \\
\( C_7 \) & 0 \\
\hline
\end{tabular}
\end{table}

The absence of \(C_7 \rtimes C_3\) and \( C_7 \) polynomials at height \( \leq 4 \) may be due to their rarity at low heights, as suggested by Foulkes \cite{foulkes}. For instance, \( C_7 \) polynomials often arise from cyclotomic fields, as shown in Table \ref{tab:C7-polynomials}, and typically appear at higher heights. To validate our database, we can apply Foulkes' resolvent method \cite{foulkes} to compute the factorization of the 30-ic and 120-ic resolvents for sampled polynomials, confirming their Galois groups against our computed GAP identifiers.
\subsection{Cyclic septic $C_7$ polynomials}
\label{sec:C7-construction}
Height-bounded scans show that the cyclic case $C_7$ is rare at small height.
In our search by projective height, the first polynomial with $Gal(f)\cong C_7$ appears at height $h=28$:
\begin{equation*}
  f(x)=x^7 + x^6 - 12x^5 - 7x^4 + 28x^3 + 14x^2 - 9x + 1.
\end{equation*}
This matches the observed scarcity of $C_7$ at low height (see Table~4 and the discussion there).
Cyclotomic constructions of cyclic septic fields using Gaussian periods usually produce polynomials with larger coefficients than typical low-height examples.
%


\section{Constructive Galois Theory}\label{sec:background}
We investigate the realization of finite groups as Galois groups of extensions of \( \Q \), employing geometric and arithmetic techniques involving branched coverings of the projective line and Hurwitz spaces. This section builds on foundational results from Serre \cite{serre} and V\"{o}lklein \cite{volklein}, with a focus on the inverse Galois problem.

\subsection{Coverings of \( \PP^1 \)}
A covering \( \phi: X \to \PP^1 \) is a finite morphism of degree \( n \) from a smooth projective curve \( X \), defined over a field \( K \) (typically \( K = \mathbb{C} \) or \( \Q \)), to the projective line \( \PP^1 \). If the cover is Galois, there exists a finite group \( G \) such that \( X \) is a \( G \)-torsor over \( \PP^1 \), with \( \PP^1 \cong X / G \). Here, \( G \) acts as the group of deck transformations, permuting the \( n \) preimages of a generic point \( t \in \PP^1(K) \). The cover is ramified at a finite set \( S = \{ p_1, \ldots, p_r \} \subset \PP^1 \), and for each \( p_i \), the inertia group at a ramification point above \( p_i \) is cyclic, generated by an element \( g_i \in G \) whose order equals the ramification index. Serre \cite[Chapter 4]{serre} establishes that such a Galois cover corresponds to a finite Galois extension \( L / K(t) \), where \( L = K(X) \) is the function field of \( X \) and \( \Gal(L / K(t)) = G \).

\subsection{Monodromy and Braid Action}
The monodromy of a Galois cover \( \phi: X \to \PP^1 \) with branch points \( \{ p_1, \ldots, p_r \} \) is encoded by a tuple \( (g_1, \ldots, g_r) \in G^r \), where \( g_i \) generates the inertia group above \( p_i \). This tuple satisfies:
\begin{enumerate}
    \item \( g_1 \cdots g_r = 1 \),
    \item \( \langle g_1, \ldots, g_r \rangle = G \).
\end{enumerate}

Such a tuple is a \textbf{Nielsen tuple}, and the set of all tuples with \( g_i \in C_i \) (for conjugacy classes \( \C = (C_1, \ldots, C_r) \)) forms the \textbf{Nielsen class} \( \text{Ni}(G, \C) \), defined by V\"{o}lklein \cite[Chapter 2, Section 2.2]{volklein} as:
\[
\text{Ni}(G, \C) = \{ (g_1, \ldots, g_r) \in C_1 \times \cdots \times C_r \mid g_1 \cdots g_r = 1, \, \langle g_1, \ldots, g_r \rangle = G \} / \text{Inn}(G),
\]
where \( \text{Inn}(G) \) denotes conjugation by elements of \( G \).

The braid group \( B_r \), with generators \( \sigma_1, \ldots, \sigma_{r-1} \) and relations:

\begin{enumerate}[label=\upshape(\roman*)]
    \item \( \sigma_i \sigma_{i+1} \sigma_i = \sigma_{i+1} \sigma_i \sigma_{i+1} \) for \( 1 \leq i \leq r-2 \),
    \item \( \sigma_i \sigma_j = \sigma_j \sigma_i \) for \( |i - j| \geq 2 \),
\end{enumerate}
acts on \( \text{Ni}(G, \C) \). The action of \( \sigma_i \) on a tuple \( (g_1, \ldots, g_r) \) is:
\[
\sigma_i \cdot (g_1, \ldots, g_r) = (g_1, \ldots, g_{i-1}, g_i g_{i+1} g_i^{-1}, g_i, g_{i+2}, \ldots, g_r).
\]
This operation, corresponding to a simple transposition of branch points \( p_i \) and \( p_{i+1} \) in the fundamental group \( \pi_1(\PP^1 \setminus \{ p_1, \ldots, p_r \}) \), preserves \( G \) and \( \C \). Serre \cite[Chapter 5, Section 5.1]{serre} leverages this action to classify covers with identical Galois groups under deformation of branch points.

\subsection{Hurwitz Spaces of Covers with Fixed Ramification Structure}
The Hurwitz space \( \mathcal{H}(G, \C) \) is a moduli space parametrizing isomorphism classes of degree-\( n \) Galois covers \( \phi: X \to \PP^1 \) over \( K \) with Galois group \( G \) and ramification type \( \C \). V\"{o}lklein \cite[Chapter 3, Section 3.1]{volklein} constructs \( \mathcal{H}(G, \C) \) as a complex variety, with points in bijective correspondence with \( \text{Ni}(G, \C) \). Its dimension is \( r - 3 \) for \( r \geq 3 \), accounting for the choice of \( r \) branch points modulo the action of \( \text{PGL}_2(K) \), which has dimension 3. A key result is the irreducibility theorem: \( \mathcal{H}(G, \C) \) is irreducible if the braid group \( B_r \) acts transitively on \( \text{Ni}(G, \C) \) \cite[Theorem 3.2]{volklein}. Serre \cite[Chapter 6, Section 6.2]{serre} shows that \( \mathcal{H}(G, \C) \) admits a model over \( \Q \), enabling arithmetic investigations of Galois realizations.

\subsection{Realizing Covers with Galois Group over \( \mathbb{C}(t) \)}
The Riemann Existence Theorem provides a cornerstone for constructing covers over \( \mathbb{C}(t) \). As stated by Serre \cite[Chapter 4, Section 4.2]{serre}, for any finite group \( G \), ramification type \( \C \), and tuple \( (g_1, \ldots, g_r) \in \text{Ni}(G, \C) \), there exists a Galois cover \( \phi: X \to \PP^1 \) over \( \mathbb{C} \) with:

\begin{enumerate}[label=\upshape(\roman*)]
    \item Branch points \( \{ p_1, \ldots, p_r \} \subset \PP^1(\mathbb{C}) \),
    \item Monodromy tuple \( (g_1, \ldots, g_r) \),
    \item Galois group \( G \).
\end{enumerate}

The extension \( \mathbb{C}(X) / \mathbb{C}(t) \) is Galois with \( \Gal(\mathbb{C}(X) / \mathbb{C}(t)) = G \), and the cover’s structure is uniquely determined by the monodromy up to isomorphism and braid group action. This theorem, rooted in the topological classification of covers via \( \pi_1(\PP^1 \setminus \{ p_1, \ldots, p_r \}) \to G \), ensures that every finite group is realizable over \( \mathbb{C}(t) \).

\subsection{Rational Points on Hurwitz Spaces and Realizing Covers over \( \Q(t) \)}
A rational point on \( \mathcal{H}(G, \C) \), when defined over \( \Q \), corresponds to a Galois cover \( \phi: X \to \PP^1 \) over \( \Q(t) \) with Galois group \( G \) and ramification type \( \C \). Such a cover arises from an irreducible polynomial \( f(x, t) \in \Q(t)[x] \) of degree \( n \), with branch points in \( \PP^1(\Q) \), satisfying \( \Gal(\Q(X) / \Q(t)) = G \). V\"{o}lklein \cite[Chapter 4, Section 4.1]{volklein} identifies the existence of such points as critical to the inverse Galois problem over \( \Q(t) \). However, the arithmetic geometry of \( \mathcal{H}(G, \C) \) imposes constraints: the action of \( \Gal(\overline{\Q} / \Q) \) on its components may obstruct rationality. Serre \cite[Chapter 6, Section 6.3]{serre} analyzes these obstructions, noting that descent to \( \Q \) requires compatibility with the Galois action on \( \text{Ni}(G, \C) \).

\subsection{Hilbert’s Irreducibility Theorem and Realizing Covers over \( \Q \)}
Hilbert’s Irreducibility Theorem bridges covers over \( \Q(t) \) to number fields. For an irreducible polynomial \( f(x, t) \in \Q(t)[x] \) defining a Galois extension \( \Q(X) / \Q(t) \) with \( \Gal(\Q(X) / \Q(t)) = G \), the theorem asserts that for most \( t_0 \in \Q \), the specialized polynomial \( f(x, t_0) \in \Q[x] \) remains irreducible with \( \Gal(f(x, t_0) / \Q) = G \). Formally, Serre \cite[Chapter 3, Section 3.1]{serre} states:

\begin{enumerate}[label=\upshape(\roman*)]
    \item Let \( f(x, t) \in \Q[t][x] \) be irreducible over \( \Q(t) \) with Galois group \( G \).
    \item The set \( \{ t_0 \in \Q \mid f(x, t_0) \text{ is reducible or } \Gal(f(x, t_0) / \Q) \neq G \} \) is a thin set in \( \Q \), i.e., contained in the image of a finite union of proper subvarieties under rational maps.
\end{enumerate}

This result, applied by V\"{o}lklein \cite[Chapter 5, Section 5.2]{volklein}, enables the construction of number fields with prescribed Galois groups by specializing rational points on \( \PP^1(\Q) \).

\subsection{Constructing Polynomials with Galois Group \( C_7 \)}\label{sec:constructing_C7}

We now turn to the problem of realizing the cyclic group \( C_7 \) of order 7 as a Galois group over \( \Q \), constructing explicit septic polynomials via the framework of constructive Galois theory. Our approach marries geometric insights from branched coverings and Hurwitz spaces with an algebraic construction using cyclotomic fields, culminating in the polynomials listed in Table \ref{tab:C7-polynomials}.

\subsubsection{Geometric Construction via Branched Coverings}
Consider a degree-7 Galois cover \( \phi: X \to \PP^1 \) with Galois group \( C_7 \). We define its ramification structure by selecting a type \( \C = (C_1, C_2, C_3, C_4) \), where each \( C_i \) is a conjugacy class in \( C_7 \). Let \( \sigma \) be a generator of \( C_7 \), represented as a 7-cycle. The non-trivial conjugacy classes of \( C_7 \) are \( \{ \sigma^k \} \) for \( k = 1, 2, 3, 4, 5, 6 \), each with cycle type (7), plus the trivial class \{1\}. We choose \( \C = (\langle \sigma \rangle, \langle \sigma^2 \rangle, \langle \sigma^4 \rangle, \langle \sigma^3 \rangle) \), corresponding to cycle types (7), (7), (7), and (3,3,1), respectively—the last indicating ramification indices 3 at two points and 1 at a third.

To determine the genus \( g \) of the curve \( X \), we apply the Riemann-Hurwitz formula:
\[
2g - 2 = n \cdot (-2) + \sum (e_i - 1),
\]
where \( n = 7 \) is the degree, and \( e_i \) are the ramification indices over the branch points. For our ramification type:

\begin{enumerate}[label=\upshape(\roman*)]
    \item Three points with \( e_1 = e_2 = e_3 = 7 \), contributing \( 3 \cdot (7 - 1) = 18 \),
    \item One point with cycle type (3,3,1), contributing \( (3 - 1) + (3 - 1) + (1 - 1) = 4 \).
\end{enumerate}

Thus:
\[
2g - 2 = 7 \cdot (-2) + (6 + 6 + 6 + 4) = -14 + 22 = 8 \implies g = 5.
\]
Hence, \( X \) is a genus-5 curve.

The Riemann Existence Theorem guarantees the existence of such a cover over \( \mathbb{C} \). Define a monodromy tuple \( (\sigma, \sigma^2, \sigma^4, \sigma^3) \in \text{Ni}(C_7, \C) \), where \( \text{Ni}(C_7, \C) \) denotes the Nielsen class of tuples generating \( C_7 \) with classes in \( \C \). Compute the product:
\[
\sigma \cdot \sigma^2 \cdot \sigma^4 \cdot \sigma^3 = \sigma^{1 + 2 + 4 + 3} = \sigma^{10} = \sigma^3,
\]
since \( \sigma^7 = 1 \). To satisfy the condition that the product equals the identity, we adjust via the braid group action, e.g., to \( (\sigma, \sigma^2, \sigma^4, \sigma^{-7}) \), yielding:
\[
\sigma \cdot \sigma^2 \cdot \sigma^4 \cdot \sigma^{-7} = \sigma^{1 + 2 + 4 - 7} = \sigma^0 = 1,
\]
while still generating \( C_7 \).

The Hurwitz space \( \mathcal{H}(C_7, \C) \) parametrizes these covers up to isomorphism. A rational point on \( \mathcal{H}(C_7, \C) \), defined over \( \Q \), corresponds to a cover \( \phi: X \to \PP^1 \) over \( \Q(t) \) with Galois group \( C_7 \), represented by an irreducible septic polynomial \( f(x, t) \in \Q(t)[x] \).

\subsubsection{Algebraic Construction via Cyclotomic Fields}
To construct explicit examples, we turn to cyclotomic fields, which naturally produce polynomials with cyclic Galois groups. The method proceeds as follows:
\begin{enumerate}
    \item \textbf{Select a prime \( p \equiv 1 \pmod{7} \)}: This ensures 7 divides \( p - 1 \). Examples include \( p = 29, 43, 71 \).
    \item \textbf{Form the cyclotomic field}: \( \Q(\zeta_p) \), where \( \zeta_p \) is a primitive \( p \)-th root of unity, has degree \( p - 1 \) over \( \Q \), with Galois group \( \Gal(\Q(\zeta_p) / \Q) \cong (\mathbb{Z}/p\mathbb{Z})^\times \), cyclic of order \( p - 1 \).
    \item \textbf{Identify the degree-7 subfield}: Since 7 divides \( p - 1 \), there exists a unique subfield \( L \subset \Q(\zeta_p) \) of degree 7, fixed by the subgroup of order \( (p - 1)/7 \).
    \item \textbf{Compute the minimal polynomial}: The minimal polynomial of a primitive element of \( L \) is an irreducible septic polynomial over \( \Q \) with Galois group \( C_7 \).
\end{enumerate}

This algebraic construction aligns with the geometric framework: the field extension \( L / \Q \) corresponds to a cover \( \phi: X \to \PP^1 \) with ramification type \( \C \), represented by a rational point on \( \mathcal{H}(C_7, \C) \).

\begin{exa} [\( p = 29 \)]
Take \( p = 29 \), where \( p - 1 = 28 \) and \( 28 / 7 = 4 \). The field \( \Q(\zeta_{29}) \) has degree 28, with a cyclic Galois group of order 28. The subgroup of order 4 fixes a subfield \( L \) of degree 7. The minimal polynomial of a primitive element of \( L \) is:
\[
f(x) = x^7 + x^6 - 12x^5 - 7x^4 + 28x^3 + 14x^2 - 9x + 1,
\]
with height 28, listed in Table \ref{tab:C7-polynomials}. Its Galois group is \( C_7 \), verified by Foulkes’ resolvent method (Table \ref{tab:foulkes_resolvents}), showing factorizations \( T(x) \to (1,1) \), \( \Psi(\psi) \to (1,7,7,7,7) \), and \( \Phi(\phi) \to (1, 17 \times 7) \).
\end{exa}

\subsubsection{Specialization and Hilbert’s Irreducibility}
Specializing the cover \( \phi: X \to \PP^1 \) over \( \Q(t) \) at a rational point \( t_0 \in \Q \) yields a polynomial \( f(x, t_0) \in \Q[x] \). Hilbert’s Irreducibility Theorem ensures that, for most \( t_0 \), \( f(x, t_0) \) remains irreducible with Galois group \( C_7 \). The polynomial for \( p = 29 \) above is such a specialization, as are the others in Table \ref{tab:C7-polynomials}, which catalog septic polynomials derived from cyclotomic subfields for various \( p \).

This synthesis of geometric and algebraic methods not only realizes \( C_7 \) over \( \Q \) but also exemplifies the power of constructive Galois theory, bridging abstract covers to tangible polynomials.

\begin{table}[h]
  \centering
  \caption{Irreducible degree-7 polynomials with Galois group \( C_7 \).}
  \label{tab:C7-polynomials}
  \begin{tabular}{|c|c|c|c|}
    \hline
    \textbf{Coefficients} & \textbf{Height} & \textbf{\( p \)} & \textbf{Galois} \\
    \hline
    \( (1,1,-12,-7,28,14,-9,1) \)            & 28       & 29   & \( C_7 \) \\ \hline
    \( (1,1,-18,-35,38,104,7,-49) \)         & 104      & 43   & \( C_7 \) \\ \hline
    \( (1,1,-30,3,254,-246,-245,137) \)      & 254      & 71   & \( C_7 \) \\ \hline
    \( (1,1,-48,37,312,-12,-49,-1) \)        & 312      & 113  & \( C_7 \) \\ \hline
    \( (1,1,-54,-31,558,-32,-1713,1121) \)   & 1713     & 127  & \( C_7 \) \\ \hline
    \( (1,1,-84,-217,1348,3988,-1433,-1163) \) & 3988   & 197  & \( C_7 \) \\ \hline
    \( (1,1,-90,69,1306,124,-5249,-4663) \)  & 5249     & 211  & \( C_7 \) \\ \hline
    \( (1,1,-102,-195,1850,978,-8933,5183) \) & 8933    & 239  & \( C_7 \) \\ \hline
    \( (1,1,-120,-711,-784,1956,2863,-343) \) & 2863    & 281  & \( C_7 \) \\ \hline
    \( (1,1,-144,399,2416,-10808,10831,-1237) \) & 10831 & 337  & \( C_7 \) \\ \hline
    \( (1,1,-162,-201,7822,12322,-107717,-193369) \) & 193369 & 379 & \( C_7 \) \\ \hline
    \( (1,1,-180,-103,6180,11596,-25209,-49213) \)   & 49213  & 421 & \( C_7 \) \\ \hline
    \( (1,1,-192,275,3952,4136,-81,-863) \)           & 4136   & 449 & \( C_7 \) \\ \hline
    \( (1,1,-198,-907,4302,20582,-18973,-56911) \)    & 56911  & 463 & \( C_7 \) \\ \hline
    \( (1,1,-210,1423,-1410,-8538,9203,19427) \)      & 19427  & 491 & \( C_7 \) \\ \hline
    \( (1,1,-234,335,13254,-42874,-55309,71879) \)    & 71879  & 547 & \( C_7 \) \\ \hline
    \( (1,1,-264,-151,13288,18556,-69425,34621) \)    & 69425  & 617 & \( C_7 \) \\ \hline
    \( (1,1,-270,116,19848,-31904,-375552,720896) \)  & 720896 & 631 & \( C_7 \) \\ \hline
    \( (1,1,-282,1345,5370,-30042,-14893,115169) \)   & 115169 & 659 & \( C_7 \) \\ \hline
    \( (1,1,-288,316,23504,-53056,-541952,1722368) \) & 1722368 & 673 & \( C_7 \) \\ \hline
    \( (1,1,-300,1631,5140,-23794,-59049,-18773) \)   & 59049  & 701 & \( C_7 \) \\ \hline
    \( (1,1,-318,-1031,26070,125148,-420841,-2302639) \) & 2302639 & 743 & \( C_7 \) \\ \hline
    \( (1,1,-324,-1483,20876,129744,36999,-54027) \)   & 129744 & 757 & \( C_7 \) \\ \hline
    \( (1,1,-354,979,30030,-111552,-715705,2921075) \) & 2921075 & 827 & \( C_7 \) \\ \hline
    \( (1,1,-378,-973,13106,-9624,-64665,91125) \)     & 91125   & 883 & \( C_7 \) \\ \hline
    \( (1,1,-390,-223,18058,30856,-116657,-225929) \)  & 225929  & 911 & \( C_7 \) \\ \hline
    \( (1,1,-408,992,48064,-204560,-1603520,8290816) \) & 8290816 & 953 & \( C_7 \) \\ \hline
    \( (1,1,-414,-4381,-10434,32702,167651,182573) \)  & 182573  & 967 & \( C_7 \) \\ \hline
  \end{tabular}
\end{table}

\subsection{Septic extensions with Galois group \(C_7 \rtimes C_3\)}
\label{sec:C7xC3-septics}

We construct infinite families of monic septics over \(\Q\) whose Galois group is
\(C_7 \rtimes C_3\).  A search by projective height found the smallest example at
height \(h=16\):
\[
f_{min}(x)=x^7-8x^5-2x^4+16x^3+6x^2-6x-2.
\]
The pattern visible here suggests a systematic approach using Chebyshev
polynomials of the first kind, defined by the recurrence
\[
T_0(x)=1,\qquad T_1(x)=x,\qquad
T_{n+1}(x)=2x\,T_n(x)-T_{n-1}(x).
\]
Explicit computation gives
\begin{align*}
T_2(x) &= 2x^2-1,\\
T_3(x) &= 4x^3-3x,\\
T_4(x) &= 8x^4-8x^2+1,\\
T_5(x) &= 16x^5-20x^3+5x,\\
T_6(x) &= 32x^6-48x^4+18x^2-1,\\
T_7(x) &= 64x^7-112x^5+56x^3-7x.
\end{align*}
Let \(S=u^2+7v^2\).  The scaled polynomial
\[
S^{7/2}\,T_7\!\bigl(x/\sqrt{S}\bigr)
 =64x^7-112S x^5+56S^2 x^3-7S^3 x
\]
has integer coefficients in \(S\) because the exponents of \(S\) are
\((7-j)/2\) for \(j=1,3,5,7\), all nonnegative integers.  Adding a constant
term \(-uS^3\) produces the two-parameter family
\[
G_7(x;u,v)=64x^7-112S x^5+56S^2 x^3-(7S^3+uS^3),
\qquad S=u^2+7v^2.
\]
For \((u,v)=(1,1)\) we obtain \(S=8\) and
\[
G_7(x;1,1)=64x^7-896x^5+3584x^3-3584x-512.
\]

\begin{lem}[Integral scaling]
\label{lem:poly-scaling}
For any indeterminate \(S\) (or nonzero element of a characteristic-zero field),
\(S^{7/2}T_7(x/\sqrt{S})\) belongs to \(\Q[S][x]\) and has degree \(7\) with leading
coefficient \(64\).
\end{lem}

\begin{proof}
Only odd powers appear in \(T_7(x)\).  Each term \(c_j x^j\) contributes
\(c_j S^{(7-j)/2} x^j\); the exponent \((7-j)/2\) is a nonnegative integer, and
\(c_7=64\).
\end{proof}

\begin{thm}
\label{thm:F21-cheb}
Define
\[
G_7(x;u,v)=
\bigl(u^2+7v^2\bigr)^{7/2}T_7\!\bigl(x/\sqrt{u^2+7v^2}\bigr)
 -u\,(u^2+7v^2)^3.
\]
\begin{enumerate}
\item Over \(\Q(u,v)\) the Galois group of \(G_7\) is \(C_7\rtimes C_3\).
\item For coprime integers \((r,s)\) with \(7\nmid rs\), all but a thin set of
specialisations \(G_7(x;r,s)\in\Q[x]\) are irreducible with the same Galois group
\(C_7\rtimes C_3\).
\end{enumerate}
Hence there are infinitely many non-isomorphic degree-7 fields over \(\Q\) whose
Galois closures have group \(C_7\rtimes C_3\).
\end{thm}

\begin{proof}
Lemma~\ref{lem:poly-scaling} confirms that \(G_7\) is a genuine septic over the
function field.  The generic Galois group \(C_7\rtimes C_3\)is the \(p=7\) instance of the
Chebyshev construction in Bruen--Jensen--Yui
\ polycite{BruenJensenYui86}.  Their effective Hilbert irreducibility theorem
guarantees that integer specialisations with \(7\nmid rs\), outside a thin set,
preserve irreducibility and the Galois group.  Distinct discriminants arising
from infinitely many such \((r,s)\) yield non-isomorphic fields.
\end{proof}

Our computation pipeline discovered 24,374  septics with Galois
group \(C_7\rtimes C_3\).  Together with the 84 septics that we have, the
collection now contains 24,458 polynomials with Galois
group \(C_7\rtimes C_3\) .
%

\section{Machine Learning Approaches for Classifying Galois Groups of Irreducible Degree-7 Polynomials}
\label{sec:ml-septics}
In this section, we describe the machine learning strategy we used to classify the Galois groups of irreducible degree-7 polynomials over $\mathbb{Q}$. Our method spans the full pipeline, from assembling the dataset to setting up baseline classifiers and then improving them by incorporating algebraic invariants. We begin with an overview of how we constructed the dataset, followed by details on our initial models, one that tackles all groups at once and another that sets aside the dominant $S_7$ cases and conclude with a refined version that draws on key features rooted in Galois theory.

To carry out these experiments, we required a large set of irreducible monic degree-7 polynomials over $\mathbb{Q}$, each tagged with its corresponding Galois group. We put this together by combining entries from the LMFDB \cite{lmfdb} with some extra polynomials that we generated ourselves. Pulling data from the LMFDB for all degree-7 number fields gave us 1,163,875 entries in total, where each one provides the minimal polynomial, discriminant, Galois group in transitive notation, and class group. The breakdown of groups in this LMFDB pull is summarized in Table~\ref{tab:lmfdb}.

\begin{table}[htb]
\centering
\caption{Galois groups in the LMFDB degree-7 data.}
\label{tab:lmfdb}
\begin{tabular}{@{}lr@{}}
\toprule
Group & Count \\
\midrule
$S_7$ & 1,062,232 \\
$A_7$ & 56,887 \\
$\mathrm{PSL}(3,2)$ & 40,861 \\
$C_7 \rtimes C_6$ & 1,547 \\
$D_7$ & 2,135 \\
$C_7 \rtimes C_3$ & 84 \\
$C_7$ & 129 \\
\bottomrule
\end{tabular}
\end{table}

Integrating the locally computed polynomials as detailed in \S\ref{sec:database}, \S\ref {sec:constructing_C7},\S\ref{sec:C7xC3-septics} increased the total to 1,686,353 unique entries. The updated distribution for this combined dataset appears in Table~\ref{tab:full}.
\begin{table}[htb]
\centering
\caption{Galois groups in the full dataset.}
\label{tab:full}
\begin{tabular}{@{}lr@{}}
\toprule
Group & Count \\
\midrule
$S_7$ & 1,559,957 \\
$A_7$ & 56,997 \\
$\mathrm{PSL}(3,2)$ & 40,977 \\
$C_7 \rtimes C_3$ & 24,457 \\
$D_7$ & 2,163 \\
$C_7 \rtimes C_6$ & 1,550 \\
$C_7$ & 252 \\
\bottomrule
\end{tabular}
\end{table}

This final dataset, named \textbf{AIMS-7}, is available at \cite{mezini2025aims7}. It includes representatives from every transitive subgroup of $S_7$ and serves as the foundation for all our training, validation, and testing.
\subsection{All-Groups Baseline Model}
\label{subsec:allgroups-coeffs}

To establish a baseline for classifying Galois groups based solely on polynomial coefficients, we first trained a model on the full AIMS-7 dataset. This allowed us to assess the information contained in the raw coefficients regarding the seven transitive subgroups listed in Table~\ref{tab:full}: $S_7$, $A_7$, $\mathrm{PSL}(3,2)$, $C_7 \rtimes C_3$, $D_7$, $C_7 \rtimes C_6$, and $C_7$.

We employed a histogram-based gradient boosting classifier, chosen for its efficiency with large datasets and numerical features. The hyper parameters included a learning rate of 0.08, a maximum of 500 iterations with early stopping after 20 unchanged iterations, and L2 regularization of $10^{-4}$. To address the severe class imbalance, where $S_7$ accounts for more than 92\% of the entries, we assigned sample weights proportional to $(\text{median count} / \text{class count}(c))^{0.5}$, thereby emphasizing the underrepresented groups without excessive bias.

The features consisted only of the coefficients, expanded into columns $a_0$ through $a_7$, with no additional preprocessing. We performed a stratified 60/40 train-test split, resulting in 1,011,811 polynomials for training and 674,542 for testing, while maintaining the original class proportions. On the training set, 5-fold stratified cross-validation yielded balanced accuracies ranging from 0.7357 to 0.7625, with a mean of $0.7526 \pm 0.0095$. These outcomes indicate that the coefficients alone convey substantial details about the Galois groups.

After retraining on the full training set, evaluation on the test set gave a balanced accuracy of 0.7726 and an overall accuracy of 0.9479. As expected given the imbalance, performance was strong for $S_7$  with precision 0.9934 and recall 0.9653. And for  $C_7 \rtimes C_3$ the precision is  0.9982 and the recall is 0.9952. Here, precision measures the proportion of correct positive predictions among all positive predictions for a class, like how accurate the model is when it predicts that group, while recall indicates the proportion of actual instances of that class that were correctly identified so how well the model finds all true cases. However, distinguishing $A_7$ and $\mathrm{PSL}(3,2)$ proved challenging: $A_7$ achieved precision 0.5314 and recall 0.8270, while $\mathrm{PSL}(3,2)$ had precision 0.4344 and recall 0.4552, with frequent misclassifications between them and toward $S_7$. For the smaller groups, the weighting improved recall at the cost of precision for instance, $C_7$ reached recall 0.8911 with precision 0.6250, $C_7 \rtimes C_6$ recall 0.6306 with precision 0.2787, and $D_7$ recall 0.6439 with precision 0.1307.

In summary, this coefficients-only baseline demonstrates that basic polynomial data provides meaningful classification power. Yet the predominance of $S_7$ encourages over prediction of that group, limiting accuracy for the rarer cases. To improve on this, we subsequently restricted attention to the non-$S_7$ polynomials and incorporated algebraic invariants.
\subsection{Non-$S_7$ Model with Algebraic Features}

For a  better separate the rarer Galois groups, we limited our work to the 126,396 polynomials from the AIMS-7 dataset without $S_7$ as the Galois group. This subset includes $A_7$, $\mathrm{PSL}(3,2)$, $C_7 \rtimes C_3$, $D_7$, $C_7 \rtimes C_6$, and $C_7$. It directly handles the big $S_7$ imbalance shown in Table~\ref{tab:full}—over 92\% of the full set. Now we can focus on the fine algebraic details that help tell these groups apart. Class sizes still differ in this subset. We kept using sample weights: $w(c) \propto (\text{median count} / \text{class count})^{0.5}$. This helps highlight small groups like $C_7$.

We used the same histogram-based gradient boosting classifier as before. Settings stayed the same: learning rate 0.08, up to 500 iterations, early stop after 20 unchanged ones, and L2 regularization of $10^{-4}$. It suits large, uneven datasets and simple number inputs well.

Features included the coefficients $a_0$ through $a_7$. We added more: the discriminant's sign and its stabilized log of the absolute value. Plus, five j-invariants $j_0$ through $j_4$. We applied a signed $\log_{10}(1 + |x|)$ to each extra feature. This controls large or small values but keeps signs.

We split the data 60/40, stratified by class. That gave 75,837 training samples and 50,559 for testing. This size works for reliable checks, even on small classes.

On the training data, 5-fold stratified cross-validation averaged a balanced accuracy of $0.8410 \pm 0.0132$. Balanced accuracy means the average recall over all classes. It gives a fair score with uneven groups. This beats the baseline and shows the algebraic features add real value.

We retrained on the full training set. The test set got a balanced accuracy of 0.8525. See Table~\ref{tab:class_report_non_s7} for details per group. Here the precision is the share of correct predictions for a group, how accurate when it picks that one. Recall is the share of true cases it finds how many real ones it catches. F1-score balances them as their harmonic mean. Scores are top-notch for $C_7 \rtimes C_3$. They're solid for cyclic and dihedral groups. But $A_7$ and $\mathrm{PSL}(3,2)$ still overlap some.

\begin{table}[htbp]
\centering
\caption{Classification Report for the Non-$S_7$ Model on the Test Set}
\label{tab:class_report_non_s7}
\begin{tabular}{lrrr}
\toprule
Galois Group & Precision & Recall & F1-Score \\
\midrule
$A_7$ & 0.7729 & 0.8850 & 0.8252 \\
$\mathrm{PSL}(3,2)$ & 0.7969 & 0.6374 & 0.7083 \\
$C_7 \rtimes C_3$ & 0.9995 & 0.9970 & 0.9983 \\
$D_7$ & 0.8274 & 0.8925 & 0.8587 \\
$C_7 \rtimes C_6$ & 0.8802 & 0.7823 & 0.8284 \\
$C_7$ & 0.9394 & 0.9208 & 0.9300 \\
\midrule
Macro Avg & 0.8694 & 0.8525 & 0.8581 \\
Weighted Avg & 0.8271 & 0.8254 & 0.8216 \\
\bottomrule
\end{tabular}
\end{table}

Table~\ref{tab:conf_matrix_non_s7} is the confusion matrix. It counts true labels against predicted ones. It shows main problems: 2,616 true $A_7$ called $\mathrm{PSL}(3,2)$, and 5,923 the other way. Outside this pair, errors are rare. This points to the model's skill at splitting off cyclic and dihedral groups. Those have tighter structures.

\begin{table}[htbp]
\centering
\caption{Confusion Matrix for the Non-$S_7$ Model on the Test Set (Rows: True, Columns: Predicted)}
\label{tab:conf_matrix_non_s7}
\begin{tabular}{lrrrrrr}
\toprule
 & $A_7$ & $\mathrm{PSL}(3,2)$ & $C_7 \rtimes C_3$ & $D_7$ & $C_7 \rtimes C_6$ & $C_7$ \\
\midrule
$A_7$ & 20,178 & 2,616 & 0 & 5 & 0 & 0 \\
$\mathrm{PSL}(3,2)$ & 5,923 & 10,447 & 1 & 19 & 1 & 0 \\
$C_7 \rtimes C_3$ & 3 & 21 & 9,754 & 3 & 0 & 2 \\
$D_7$ & 2 & 20 & 2 & 772 & 65 & 4 \\
$C_7 \rtimes C_6$ & 2 & 2 & 0 & 131 & 485 & 0 \\
$C_7$ & 0 & 3 & 2 & 3 & 0 & 93 \\
\bottomrule
\end{tabular}
\end{table}

For completeness, we also repeated the same pipeline with alternative
train--test splits of $50/50$ and $80/20$.  In classical machine learning
practice, an $80/20$ split is often the default choice, since it gives the
model as much training data as possible while still keeping a separate test
set.  In our case, all three splits produced very similar test balanced
accuracies, in a narrow range around $0.85$--$0.87$, with a slight gain when
more data was allocated to training.  We chose to report detailed results for
the $60/40$ split, which still follows this intuition that most of the data is
used for learning,  but it leaves a particularly large and stratified test set.
 The full training and evaluation pipeline is available in our public
repository~\cite{mezini2025mlgalois7repo}.

\section{Symbolic Neural Networks for Galois Group Classification}
\label{sec:symbolic-nn}

Our machine learning experiments pointed to a small set of features, the polynomial coefficients, aspects of the discriminant, and certain algebraic invariants, that prove especially useful in identifying Galois groups. Building on that observation, we now explore the idea of a symbolic neural network, which brings these features front and center in the model's design. Traditional neural networks deliver good results, yet they can feel opaque when it comes to how they incorporate mathematical knowledge. A symbolic network, on the other hand, pairs learning from data with explicit symbolic computations, offering greater clarity while drawing directly on what we know from the field.

These kinds of networks, often termed neuro-symbolic systems, aim to merge statistical patterns with logical structures. In disciplines like algebra and number theory, where symmetries and invariants are key, this combination makes a lot of sense. Classifying Galois groups isn't just a matter of crunching numbers; it involves recognizing the underlying symmetries in the roots. By hardwiring calculations for things like discriminants and resolvents, the network creates representations that tie back to core theory. 

Focusing on our irreducible monic septics over the rationals, the boosting model outcomes highlight the discriminant's sign, which is vital for even-permutation groups when square, its scaled logarithm, and modified $j_0$ to $j_4$ invariants as essential for sorting out confusions, like those between $A_7$ and $\mathrm{PSL}(3,2)$, or among cyclic and dihedral types. This shaped how we structured things: begin with pulling invariants algebraically from the coefficients, such as handling the discriminant and its features, then apply scaling via signed logs to keep numbers manageable. From there, layers mix in the raw data, perhaps using attention to pick out what's most relevant for group differences.

Putting it together, the input comes as the coefficient vector $(a_0, \dots, a_6, 1)$ for the monic form, including that leading one for an even eight entries. The network flows like this:

\begin{enumerate}[label=\arabic*.]
\item \textbf{Symbolic Preprocessing Layer}: Extracts core invariants through fixed formulas, the discriminant and $j_0$ to $j_4$ as septic-specific resolvents. Out comes a vector of these untreated values, grounding the model in quantities unchanged by root reordering.

\item \textbf{Feature Transformation Layer}: To cope with vast value ranges, each $x$ gets mapped to $\operatorname{sgn}(x) \cdot \log_{10}(1 + |x|)$, keeping signs for important qualities like square discriminants and taming sizes for reliable training.

\item \textbf{Concatenation Layer}: Blends the initial coefficients with these adjusted invariants into a 15-entry vector (eight coefficients, two discriminant parts, five tuned $j_i$). This setup allows balancing straightforward polynomial traits against theory-derived ones.

\item \textbf{Dense Neural Layers}: Two connected layers with 128 nodes, then 64, with ReLU each time, tease out complex nonlinear links, leaning on the symbolic base to spot what sets groups apart.

\item \textbf{Attention Layer}: A self-attention setup dynamically emphasizes features, say lifting the discriminant for parity tests or specific $j_i$ for subgroup clues, honing in on the essentials.

\item \textbf{Output Layer}: Softmax across six nodes yields probabilities for non-$S_7$ groups: $A_7$, $\mathrm{PSL}(3,2)$, $C_7 \rtimes C_3$, $D_7$, $C_7 \rtimes C_6$, and $C_7$. We train with weighted cross-entropy to level the playing field for scarcer groups, in line with past methods.
\end{enumerate}

Training sticks to our earlier approach of stratified splits and balancing, but the symbolic elements open doors for insight perhaps tracing attention flows to see which invariants sway decisions in tricky spots like $A_7$ against $\mathrm{PSL}(3,2)$.

In essence, this framework extends our boosting work while boosting readability, applied to the AIMS-7 data, it could offer rough symbolic shortcuts for classifications, possibly prompting fresh ideas in theory. It's still in early stages, but layering on features like Frobenius cycles modulo primes might make it a go-to helper for algebraic studies. Further out, the approach could reach into higher degrees, using symbolic anchors to navigate denser Galois landscapes.

\section{Extending to Higher-Degree Polynomials}
\label{sec:conclusions}
The work presented here centers on irreducible monic polynomials of degree~7 over $\mathbb{Q}$, a range where the transitive subgroups of $S_7$ are fully cataloged, Galois groups can be computed reliably, and invariants such as $j_0,\dots,j_4$ are available in explicit form. This setting allows us to build and test the symbolic neural network in Section~\ref{sec:symbolic-nn} with some confidence, using computer algebra to check every label and every invariant.

Pushing to higher degrees reveals clear limits on both the Galois and the invariant side. For Galois computations, methods \cite{cohen} have no theoretical degree cap and are implemented in systems such as PARI/GP, and SageMath \cite{sagemath}. However, PARI/GP and SageMath which uses PARI/GP by default, that support this computation for polynomials up to degree 11 . It will require the galdata package for degrees 8–11, which is included in SageMath. For degrees higher than 11, SageMath does not have a built-in open-source method to reliably identify the transitive group without external software like Magma.  Attempts here may fail due to computational limits, such as memory or time constraints when trying to compute the Galois closure explicitly. In practice, the running time and memory usage increase rapidly with the degree and with the size of the discriminant. Computing $\mathrm{Gal}(f/\mathbb{Q})$ for large random polynomials of degree $n \ge 10$ is already costly, and scaling up to build datasets on the scale of AIMS-7 in degrees $12$ or $15$ would require substantial resources. Beyond that, routine checks on large families of random polynomials become impractical.

Invariant theory faces similar issues. For binary forms, explicit minimal generating sets for the invariant rings are known in degrees up to $10$; see Curri \cite{curri2021} and the references therein for a unified account of the cases $3 \le d \le 10$. For larger $d$ one still has Hilbert’s finiteness theorem, but no comparably concrete list of generators is available. In that regime the clean flow
\[
  \text{coefficients} \;\longrightarrow\; \text{explicit invariants} \;\longrightarrow\; \text{symbolic network}
\]
breaks down.

These constraints, however, suggest a direction rather than an endpoint. Degree~7 serves as a controlled proving ground, guiding extensions to degrees $8$ through $10$, where invariants are still accessible and Galois algorithms remain viable. One can imagine training tailored networks on carefully chosen families in these degrees and tracking which combinations of coefficients and low-degree invariants are actually used in the classification.

For higher degrees, where labelled data are sparse and full invariant rings are not known explicitly, the role of the network changes from classifier to exploratory tool. By examining the weights on coefficient patterns or on approximate invariants, one may be able to pick out expressions that consistently separate groups in examples and therefore deserve a theoretical explanation. The aim is not to replace Galois group algorithms, but to supplement them with heuristic guidance. With combinations of coefficients that behave like resolvents for certain subgroups and that might eventually be understood in purely algebraic terms.

In short, septics are not the end of the story, they are a launchpad. Degree~7 is one of the last places where both the Galois side and the invariant side are still explicit enough that one can attach a clear algebraic meaning to each layer of a symbolic network. Extending even modestly beyond this range could help connect heavy computational tools with broader classification results, and turn the present limitations into questions for future work.

\subsection{Statistical analysis of the non-\texorpdfstring{$S_7$}{S7} database}

In this subsection we keep the same non-$S_7$ database and notation as above.
All plots are functions of the logarithmic height $h(f)=\log_{10}(H(f)+1)$
and $\log_{10}|\Delta(f)|$.

\subsubsection{Distribution of Galois groups}

The overall class frequencies are shown in
Figure~\ref{fig:group-distribution-pie}.  One immediately sees that the
simple groups $A_7$ and $\mathrm{PSL}(3,2)$ occupy most of the disk, while
the smaller groups $C_7$, $C_7\rtimes C_3$, $C_7\rtimes C_6$ and $D_7$
contribute only thin slices.  Thus, even after removing the $S_7$-part, the
dataset is strongly biased toward the two “generic’’ simple groups.

\begin{figure}[t]
  \centering
  \includegraphics[width=0.55\textwidth]{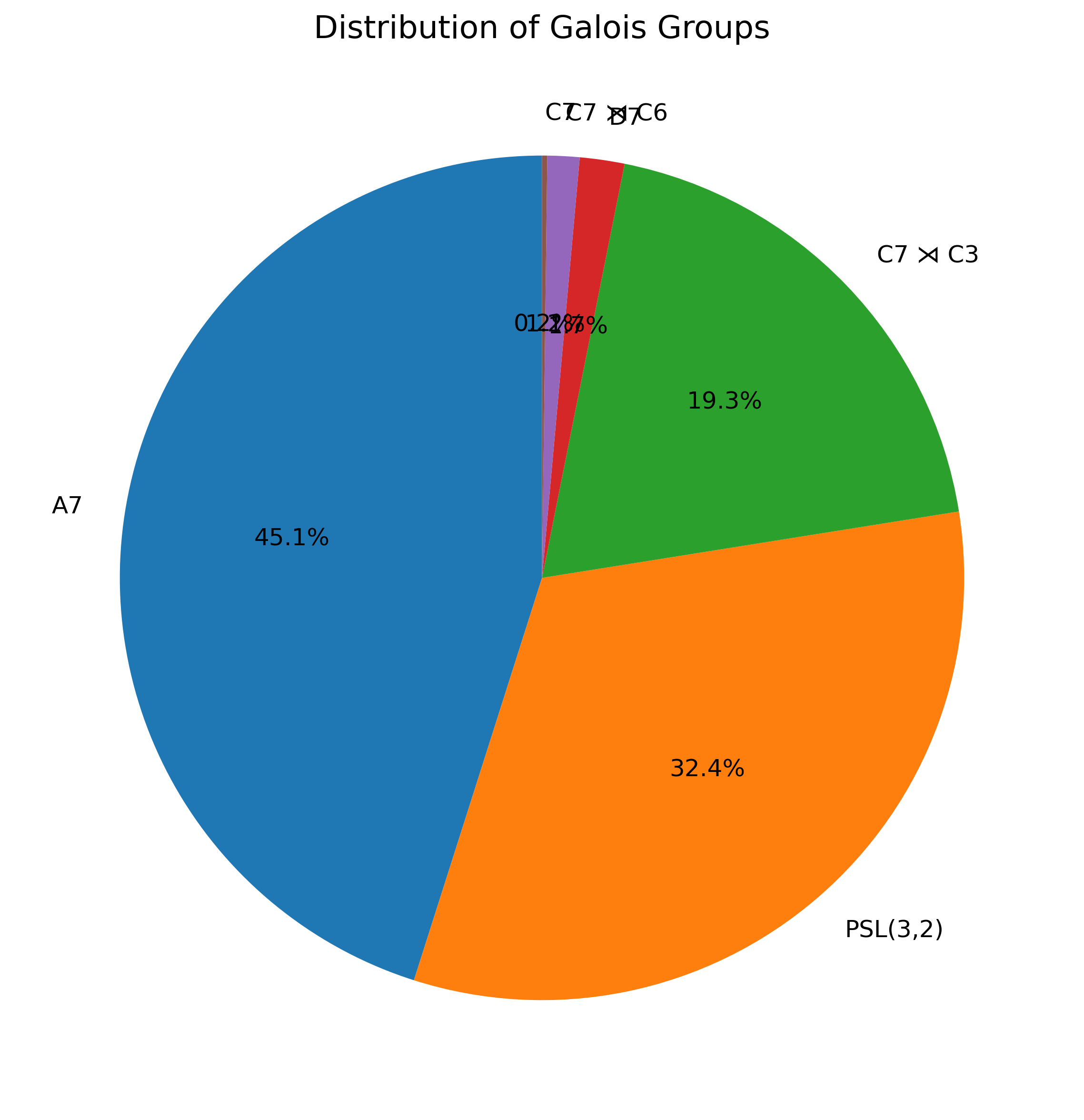}
  \caption{Relative frequencies of the six non-$S_7$ Galois groups.}
  \label{fig:group-distribution-pie}
\end{figure}

\subsubsection{Log-height distributions by group}

The empirical distributions of $h(f)$ for each Galois group are given in
Figure~\ref{fig:height-dist}.  The six coloured histograms show that the
groups live on very different horizontal ranges.  The distributions for
$A_7$ and $\mathrm{PSL}(3,2)$ are narrow and concentrated between roughly
$h\approx 1$ and $h\approx 3$, whereas the distributions for $C_7$ and
$C_7\rtimes C_3$ are shifted far to the right and have long tails extending
beyond $h\approx 10$.

The same phenomenon is visible in the boxplots of
Figure~\ref{fig:height-boxplots}.  Here the vertical displacement of the
boxes encodes the differences between groups: the median log-height for
$C_7\rtimes C_3$ lies high above that of $A_7$, and the interquartile ranges
do not overlap.  In particular, the order of the medians is
\[
  C_7\rtimes C_3 \;\gg\; C_7 \;>\; D_7 \approx C_7\rtimes C_6
  \;>\; \mathrm{PSL}(3,2) \gtrsim A_7.
\]
Figure~\ref{fig:mean-height-by-group} summarises this information in a
single bar plot: the mean log-height $\mathbb{E}[h(f)\mid G]$ decreases
steadily as we move from $C_7\rtimes C_3$ and $C_7$ to $A_7$.

\begin{figure}[t]
  \centering
  \includegraphics[width=0.8\textwidth]{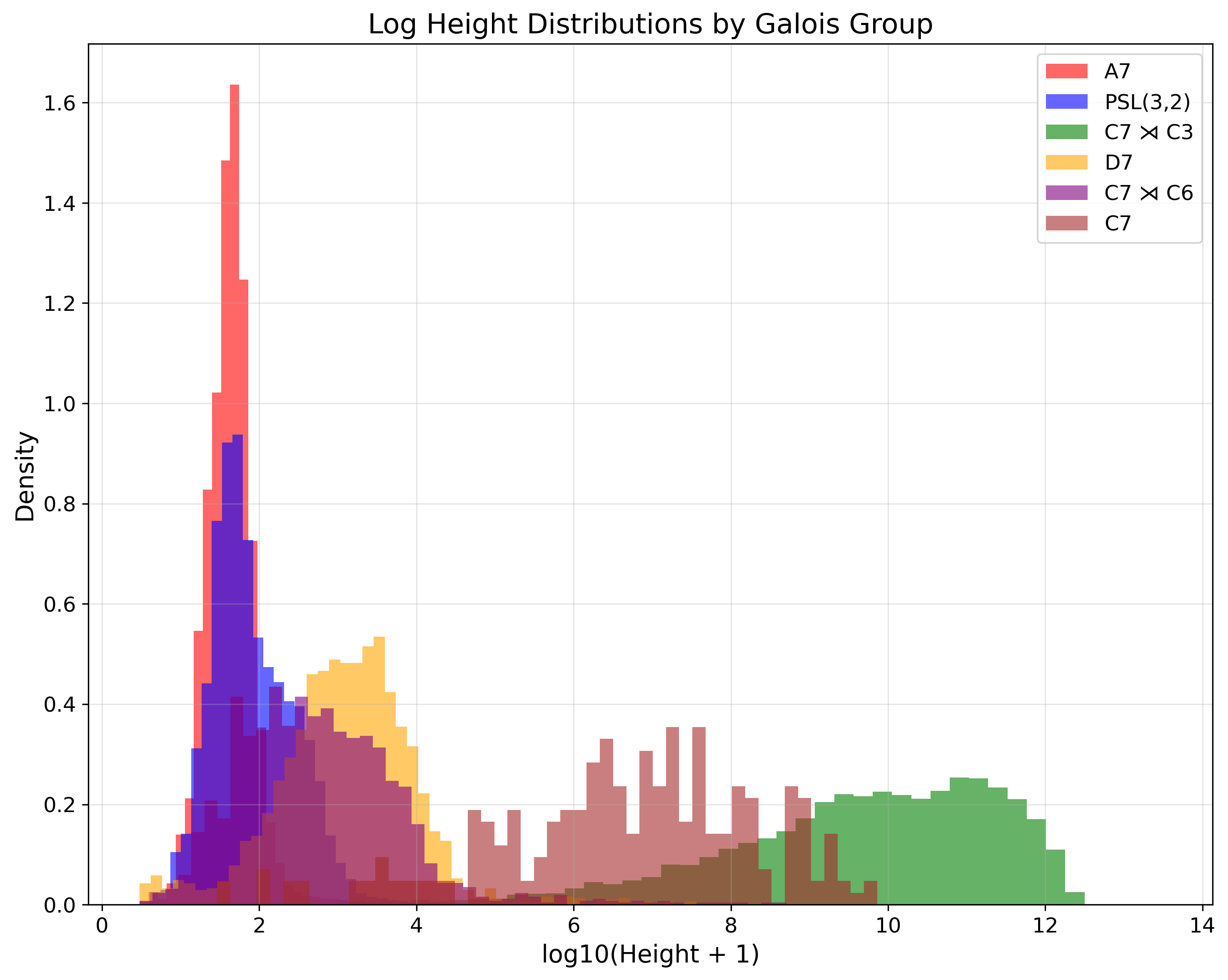}
  \caption{Empirical distributions of the log-height $h(f)$ for each
    non-$S_7$ Galois group.}
  \label{fig:height-dist}
\end{figure}

\begin{figure}[t]
  \centering
  \includegraphics[width=0.6\textwidth]{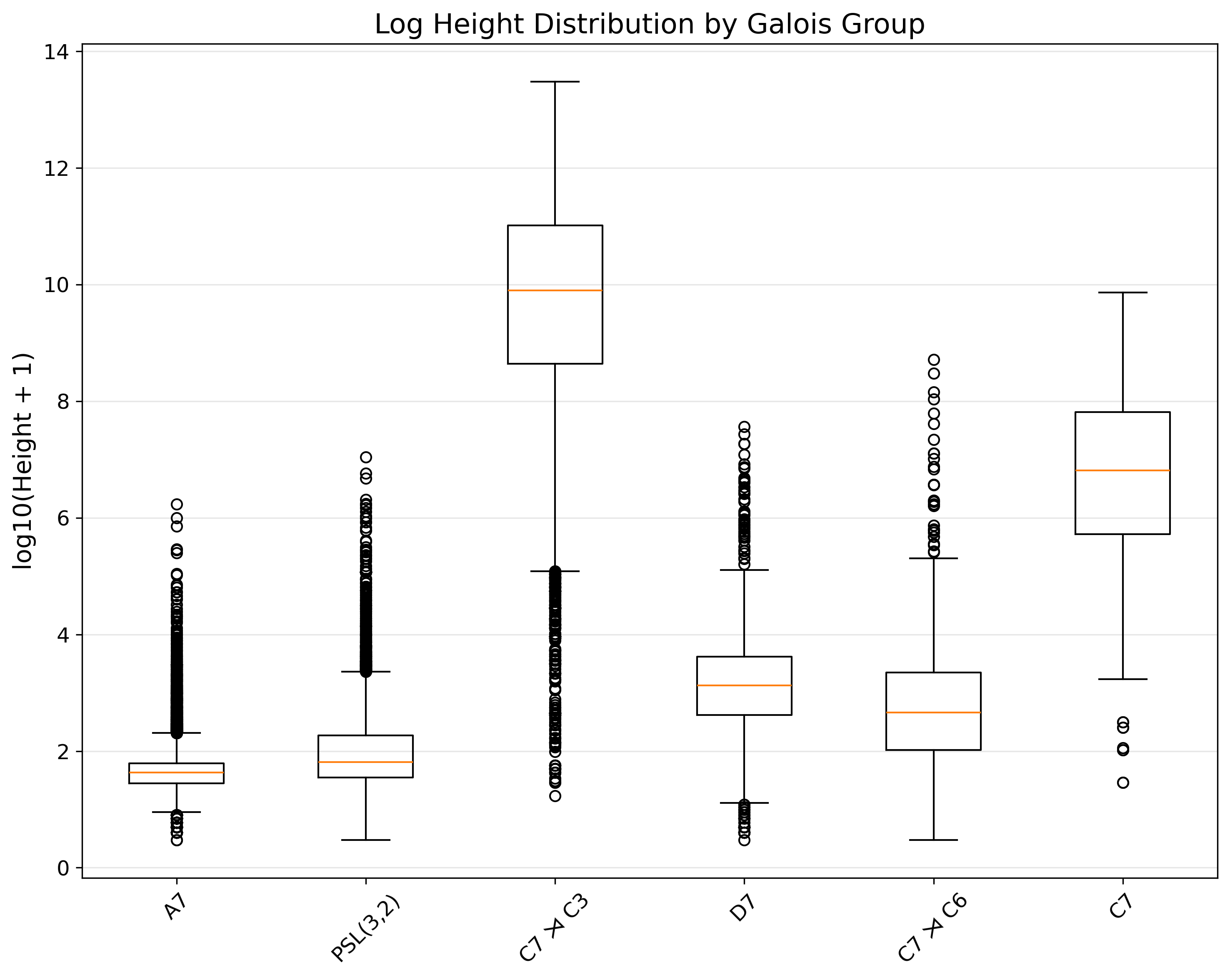}
  \caption{Boxplots of $h(f)$ by Galois group.  Each group occupies its own
   characteristic height range.}
  \label{fig:height-boxplots}
\end{figure}

\begin{figure}[t]
  \centering
  \includegraphics[width=0.6\textwidth]{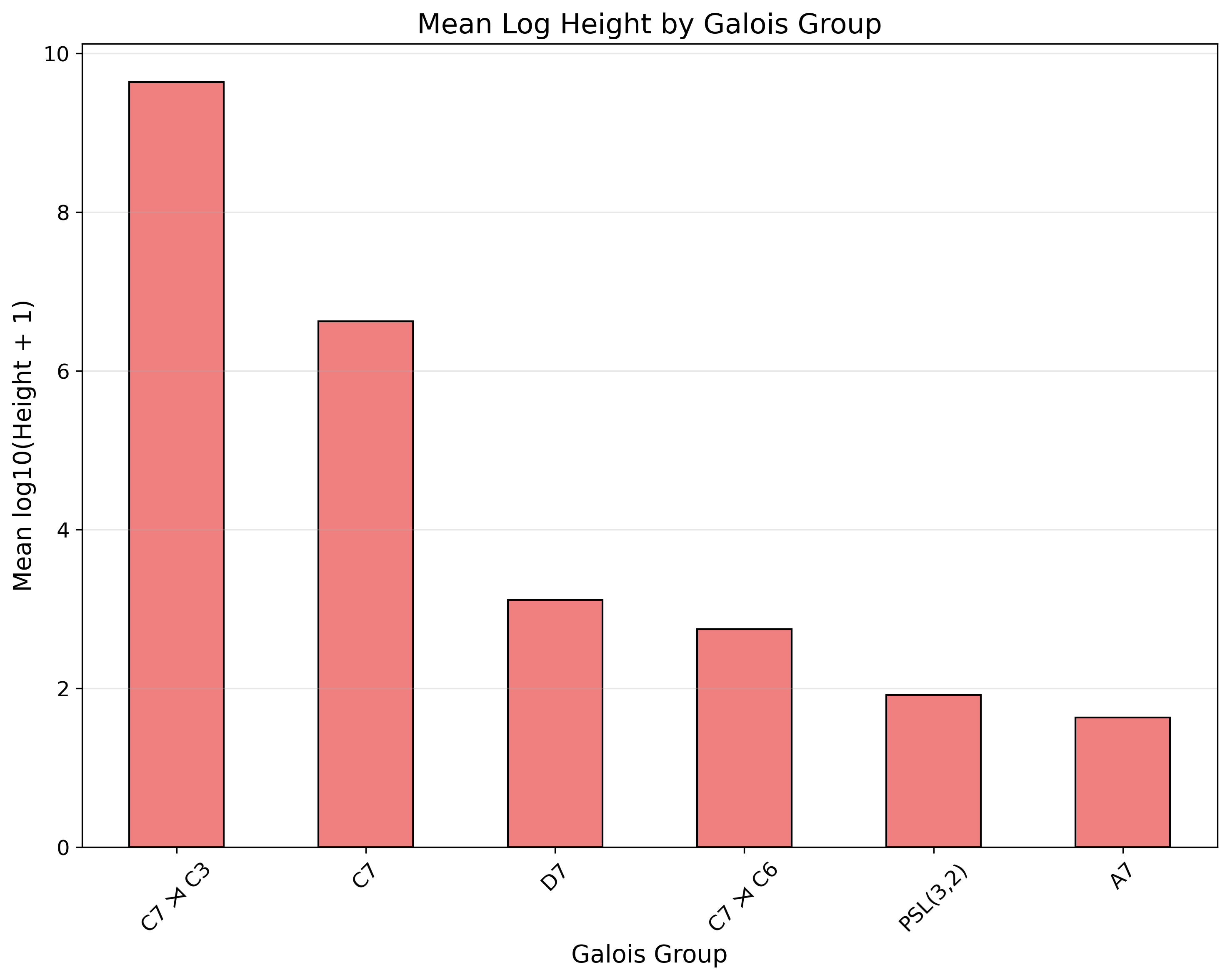}
  \caption{Mean log-height $\mathbb{E}[h(f)\mid G]$ for each
    non-$S_7$ Galois group.}
  \label{fig:mean-height-by-group}
\end{figure}

\subsubsection{Log-height versus group complexity}

To place the six groups on a single horizontal scale we use the complexity
index $c(G)\in\{1,\dots,6\}$ defined earlier, ordered from $C_7$ up to $A_7$.
Figure~\ref{fig:height-vs-complexity-scatter} shows the point cloud
\[
   (c(G(f)),\, h(f))
\]
for all polynomials in the database.  Each vertical column corresponds to
one group.  The tallest columns occur at $c(G)=1$ and $2$ (for $C_7$ and
$C_7\rtimes C_3$), and the columns get visibly shorter as one moves towards
$c(G)=5$ and $6$ (for $\mathrm{PSL}(3,2)$ and $A_7$).

The mean behaviour is extracted in
Figure~\ref{fig:mean-height-vs-complexity}, where we plot the average
log-height at each complexity level.  The curve starts high at $c=1,2$,
drops sharply at $c=3$, and then decreases slowly up to $c=6$.  Thus, in our
data, greater group complexity is associated with smaller defining height,
which is the opposite of the naive expectation that “more complicated groups
should require more complicated polynomials.’’

\begin{figure}[t]
  \centering
  \includegraphics[width=0.6\textwidth]{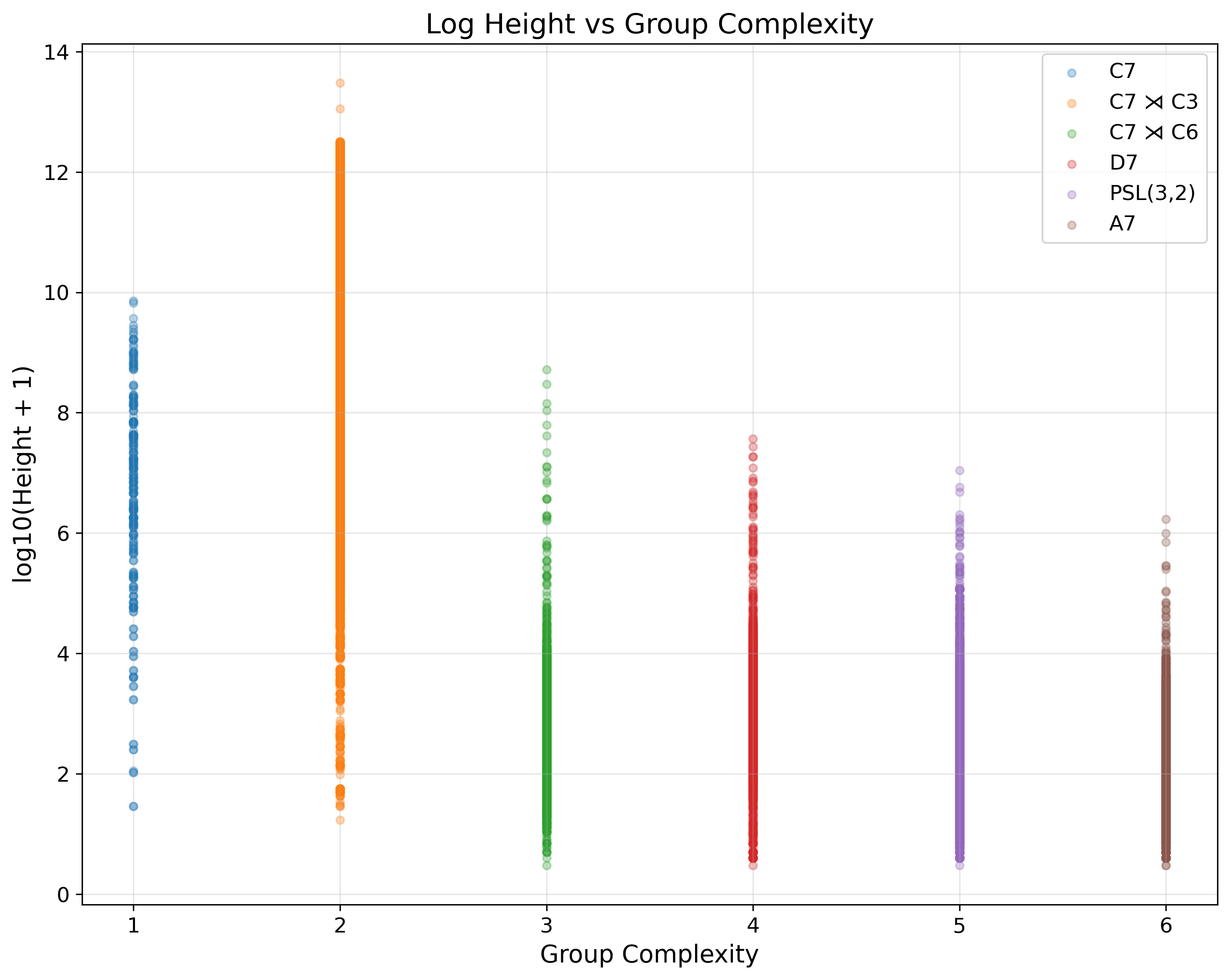}
  \caption{Scatter plot of $h(f)$ versus the complexity index $c(G)$.
   Each vertical strip corresponds to one Galois group.}
  \label{fig:height-vs-complexity-scatter}
\end{figure}

\begin{figure}[t]
  \centering
  \includegraphics[width=0.6\textwidth]{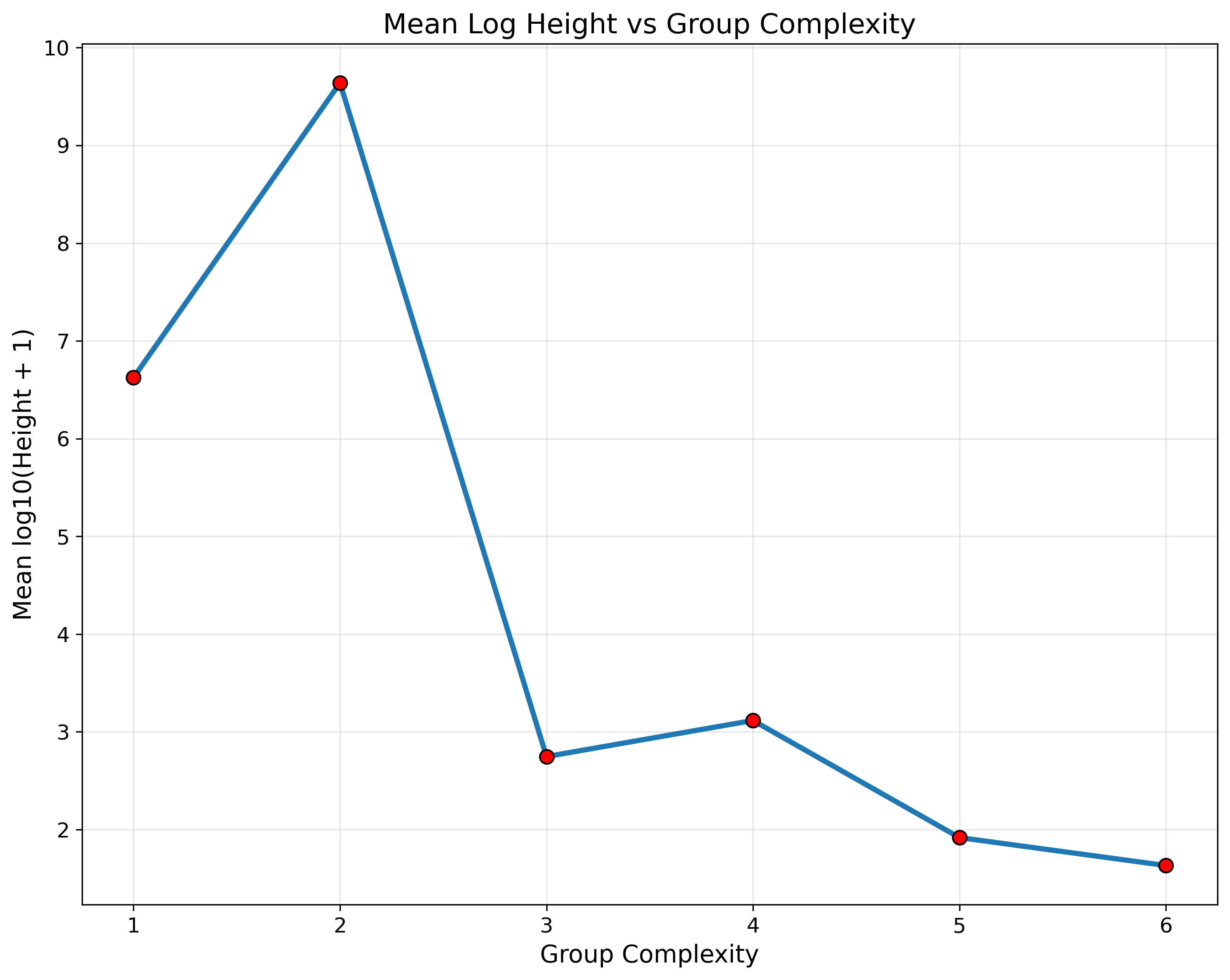}
  \caption{Mean log-height as a function of $c(G)$.  The decreasing curve
   shows that higher group complexity is realised by smaller heights in this
   database.}
  \label{fig:mean-height-vs-complexity}
\end{figure}

\subsubsection{Discriminant statistics and correlation with height}

The histogram in Figure~\ref{fig:disc-dist} shows the distribution of
$\log_{10}|\Delta(f)|$ for the non-$S_7$ database.  Most mass lies between
$10$ and $20$, with a visible tail of polynomials whose discriminant
magnitude is several orders of magnitude larger.

Finally, Figure~\ref{fig:height-vs-disc} plots the pairs
\[
  \bigl(h(f),\, \log_{10}|\Delta(f)|\bigr)
\]
and colours each point by its Galois group.  The cloud has a clear diagonal
shape: large heights tend to come with large discriminants.  The colour
structure matches what we observed before: the $C_7$ and $C_7\rtimes C_3$
points populate the region with both large height and large discriminant,
whereas the $A_7$ and $\mathrm{PSL}(3,2)$ points are concentrated near the
lower-left corner.

\begin{figure}[h]
  \centering
  \includegraphics[width=0.6\textwidth]{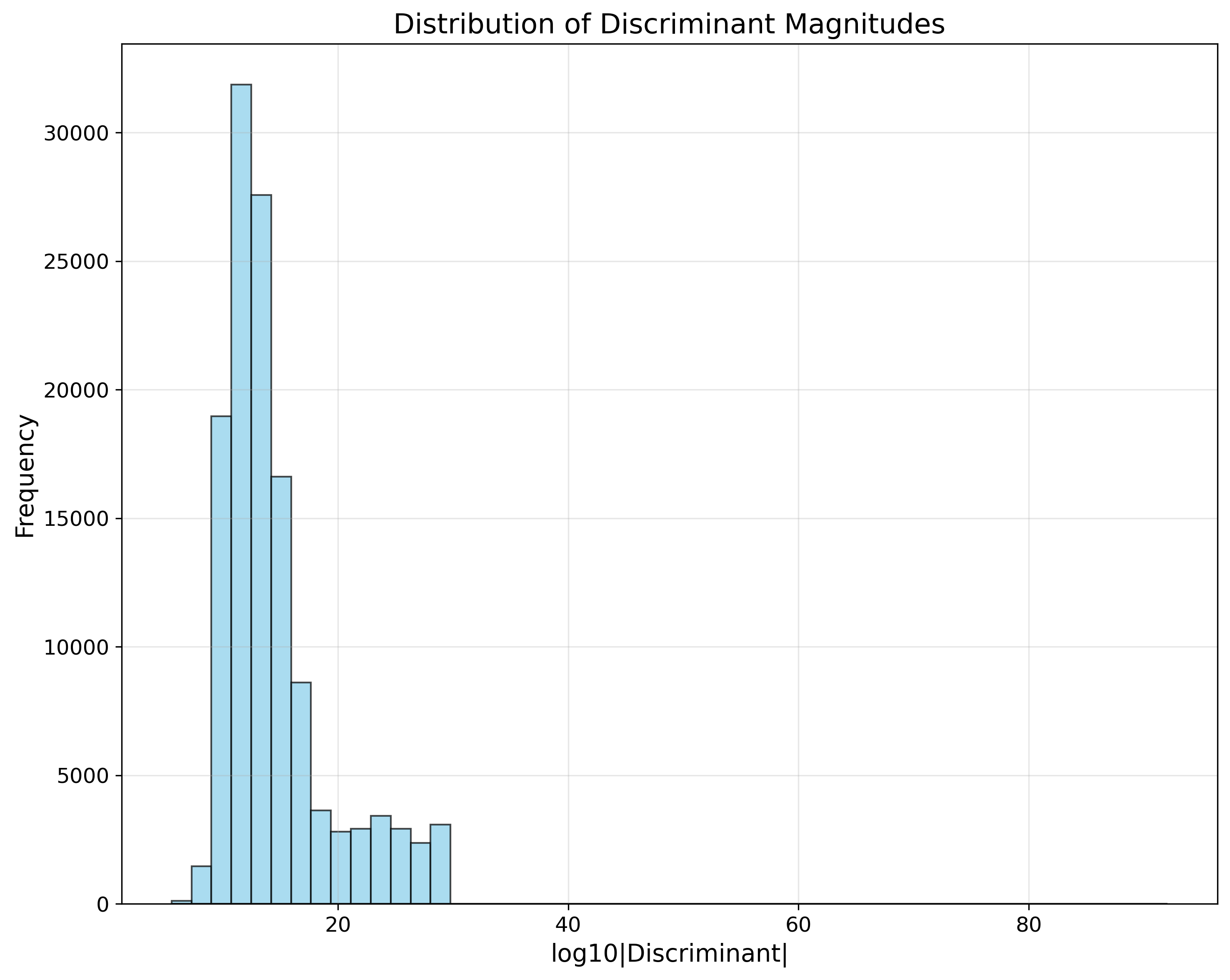}
  \caption{Histogram of $\log_{10}|\Delta(f)|$ for the non-$S_7$ database.}
  \label{fig:disc-dist}
\end{figure}

\begin{figure}[h]
  \centering
  \includegraphics[width=0.6\textwidth]{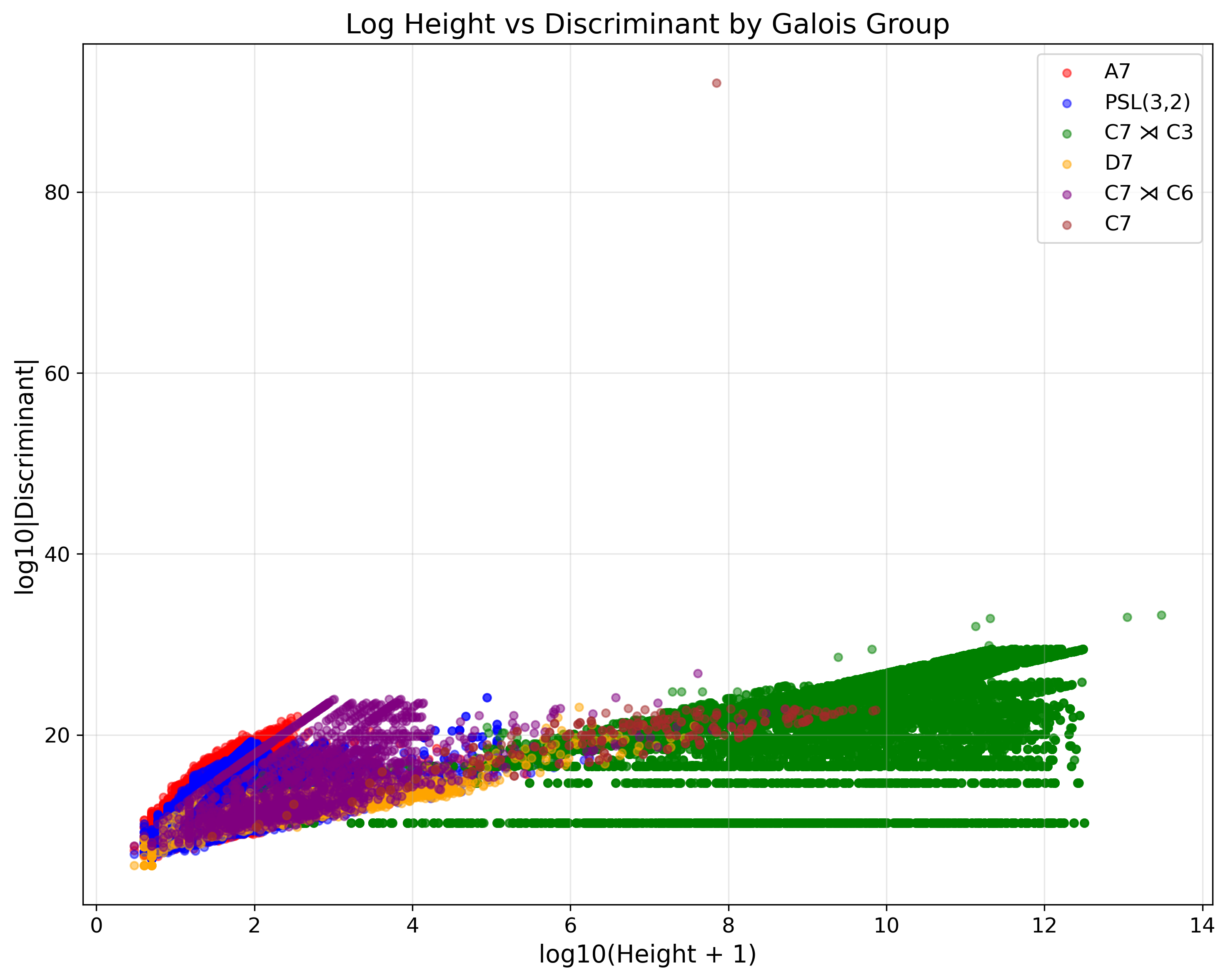}
  \caption{Joint plot of $h(f)$ versus $\log_{10}|\Delta(f)|$, coloured by
   Galois group. Larger heights tend to be accompanied by larger
   discriminants.}
  \label{fig:height-vs-disc}
\end{figure}
All statistical plots and clustering analyses in this section were produced
with our public implementation~\cite{mezini2025statsgalois7repo}.

\section{Concluding Remarks}

This work introduces a scalable computational framework for the classification of Galois groups associated to irreducible degree-7 polynomials over~$\Q$. By combining classical tools such as resolvent factorizations and invariant theory with modern data-driven methods, we construct a reproducible dataset of over one million septics and demonstrate that algebraic invariants derived from transvections can significantly improve classification performance, particularly for rare solvable groups.

The resulting neurosymbolic model illustrates the effectiveness of hybrid approaches that integrate symbolic algebraic features with supervised learning, offering improved interpretability over purely black-box classifiers. The empirical analysis highlights both the utility and limitations of current methods, especially under class imbalance induced by the natural distribution of Galois groups under height constraints.

Several avenues for further research emerge from this study. First, the methods developed here can be extended to polynomials of higher degrees, where the classification problem becomes substantially more complex and current resolvent constructions may require further refinement or approximation. Second, the structure and distribution of Galois groups in families defined by different height or shape conditions remains an open area of empirical investigation. Third, the combination of this framework with arithmetic geometry particularly the study of moduli spaces and Hurwitz covers may enable more systematic realizations of specific groups as Galois groups over~$\Q$.

From a computational perspective, future work may explore the integration of unsupervised learning techniques, the generation of synthetic rare group examples, and the refinement of invariant-based feature engineering. The database constructed in this paper provides a foundation for such extensions and may also serve as a benchmark for future work on the Inverse Galois Problem in both theoretical and experimental contexts.

\bibliographystyle{amsplain}
\bibliography{2025-1.bib}

\appendix
\section{Coefficients of the degree 35 resolvent}\label{app:a}
\begin{tiny}
\[
\begin{split}
e_1 &= -15 a_6, \\
e_2 &=   105 a_6^2 + 10 a_5, \\
e_3 &=    = -455 a_6^3 + 150 a_6 a_5 + \frac{89}{3} a_4, \\
e_4 & = 1365 a_6^4 - 980 a_6^2 a_5 - \frac{80}{3} a_5^2 - \frac{596}{3} a_6 a_4 + \frac{56}{3} a_3, \\
e_5 &=   -4095 a_6^5 + 3675 a_6^3 a_5 + 200 a_6^2 a_4 - 900 a_6 a_5^2 - 56 a_5 a_4 - \frac{2972}{5} a_6 a_3 + \frac{356}{5} a_2, \\
e_6 &=  12285 a_6^6 - 14175 a_6^4 a_5 - 910 a_6^3 a_4 + 5775 a_6^2 a_5^2 + 672 a_6 a_5 a_4 + 80 a_5^3 + 1660 a_6^2 a_3 \\
& - 672 a_5 a_3 - 252 a_6 a_2 + \frac{1068}{5} a_1, \\
e_7 &=   = -36855 a_6^7 + 49665 a_6^5 a_5 + 3185 a_6^4 a_4 - 28350 a_6^3 a_5^2 - 1820 a_6^2 a_5 a_4 - 700 a_6 a_5^3 \\
& - 6146 a_6^3 a_3 + 2352 a_6 a_5 a_3 + 756 a_5^2 a_4 + 756 a_6^2 a_2 - 1512 a_5 a_2 - 504 a_6 a_1 + 252 a_0, \\
e_8 & = 110565 a_6^8 - 171990 a_6^6 a_5 - 11102 a_6^5 a_4 + 127575 a_6^4 a_5^2 + 6370 a_6^3 a_5 a_4 + 4550 a_6^2 a_5^3 \\
& + 22386 a_6^4 a_3 - 8190 a_6^2 a_5 a_3 - 2520 a_6 a_5^2 a_4 - 336 a_5^2 a_3 - 2646 a_6^3 a_2 + 6048 a_6 a_5 a_2 \\
& + 672 a_5 a_4 a_3 + 1008 a_6^2 a_1 - 2016 a_5 a_1 - 672 a_6 a_0 + 336 a_4, \\
e_9 &= -331695 a_6^9 + 589680 a_6^7 a_5 + 38857 a_6^6 a_4 - 552825 a_6^5 a_5^2 - 22295 a_6^4 a_5 a_4 - 22750 a_6^3 a_5^3 \\
& - 83265 a_6^5 a_3 + 28665 a_6^3 a_5 a_3 + 8820 a_6^2 a_5^2 a_4 + 1176 a_5^3 a_4 + 9261 a_6^4 a_2 - 21168 a_6^2 a_5 a_2 \\
& - 2352 a_6 a_5^2 a_3 - 2352 a_5^2 a_2 - 3528 a_6^3 a_1 + 7056 a_6 a_5 a_1 + 2352 a_6^2 a_0 - 1176 a_5 a_0 - 1176 a_4 a_3, \\
e_{10} &= 995085 a_6^{10} - 1986210 a_6^8 a_5 - 136059 a_6^7 a_4 + 2284590 a_6^6 a_5^2 + 77945 a_6^5 a_5 a_4 + 108745 a_6^4 a_5^3 \\
& + 299691 a_6^6 a_3 - 100275 a_6^4 a_5 a_3 - 30870 a_6^3 a_5^2 a_4 - 4116 a_6^2 a_5^3 a_3 - 32487 a_6^5 a_2 + 74172 a_6^3 a_5 a_2 \\
& + 8232 a_6^2 a_5^2 a_3 + 8232 a_6 a_5^3 a_2 + 12348 a_6^4 a_1 - 24696 a_6^2 a_5 a_1 - 8232 a_6^3 a_0 + 4116 a_6 a_5 a_0 + 4116 a_4 a_2, \\
e_{11} &= -2985255 a_6^{11} + 6552150 a_6^9 a_5 + 476721 a_6^8 a_4 - 9088350 a_6^7 a_5^2 - 272657 a_6^6 a_5 a_4 \\
& - 503685 a_6^5 a_5^3 - 1049565 a_6^7 a_3 + 351232 a_6^5 a_5 a_3 + 108045 a_6^4 a_5^2 a_4 + 14406 a_6^3 a_5^3 a_4 \\
& + 113695 a_6^6 a_2 - 259602 a_6^4 a_5 a_2 - 28824 a_6^3 a_5^2 a_3 - 28824 a_6^2 a_5^3 a_2 - 43236 a_6^5 a_1 \\
& + 86472 a_6^3 a_5 a_1 + 28824 a_6^4 a_0 - 14412 a_6^2 a_5 a_0 - 14412 a_4 a_1, \\
e_{12} &= 8955765 a_6^{12} - 21290850 a_6^{10} a_5 - 1670523 a_6^9 a_4 + 34959450 a_6^8 a_5^2 + 954799 a_6^7 a_5 a_4 \\
& + 2277735 a_6^6 a_5^3 + 3673455 a_6^8 a_3 - 1228292 a_6^6 a_5 a_3 - 378162 a_6^5 a_5^2 a_4 - 50421 a_6^4 a_5^3 a_4 \\
& - 398432 a_6^7 a_2 + 909107 a_6^5 a_5 a_2 + 100884 a_6^4 a_5^2 a_3 + 100884 a_6^3 a_5^3 a_2 + 151326 a_6^6 a_1 \\
& - 302652 a_6^4 a_5 a_1 - 100884 a_6^5 a_0 + 50442 a_6^3 a_5 a_0 + 50442 a_4 a_0, \\
e_{13} &= -26867295 a_6^{13} + 68068350 a_6^{11} a_5 + 5851833 a_6^{10} a_4 - 131135850 a_6^9 a_5^2 - 3345797 a_6^8 a_5 a_4 \\
& - 10110195 a_6^7 a_5^3 - 12857025 a_6^9 a_3 + 4299992 a_6^7 a_5 a_3 + 1323567 a_6^6 a_5^2 a_4 + 176475 a_6^5 a_5^3 a_4 \\
& + 1394512 a_6^8 a_2 - 3183877 a_6^6 a_5 a_2 - 353094 a_6^5 a_5^2 a_3 - 353094 a_6^4 a_5^3 a_2 - 529641 a_6^7 a_1 \\
& + 1059282 a_6^5 a_5 a_1 + 353094 a_6^6 a_0 - 176547 a_6^4 a_5 a_0 - 176547 a_3 a_0, \\
e_{14} &= 80601885 a_6^{14} - 214693950 a_6^{12} a_5 - 20481423 a_6^{11} a_4 + 485099250 a_6^{10} a_5^2 + 11709787 a_6^9 a_5 a_4 \\
& + 44244225 a_6^8 a_5^3 + 44999575 a_6^{10} a_3 - 15049972 a_6^8 a_5 a_3 - 4632782 a_6^7 a_5^2 a_4 - 617703 a_6^6 a_5^3 a_4 \\
& - 4880792 a_6^9 a_2 + 11143567 a_6^7 a_5 a_2 + 1235829 a_6^6 a_5^2 a_3 + 1235829 a_6^5 a_5^3 a_2 + 1853743 a_6^8 a_1 \\
& - 3707486 a_6^6 a_5 a_1 - 1235829 a_6^7 a_0 + 617914 a_6^5 a_5 a_0 + 617914 a_2 a_0 \\
e_{15} &= -241805655 a_6^{15} + 671458650 a_6^{13} a_5 + 71694981 a_6^{12} a_4 - 1771058250 a_6^{11} a_5^2 - 40984277 a_6^{10} a_5 a_4 \\ 
&- 190672575 a_6^9 a_5^3 - 157498512 a_6^{11} a_3 + 52674932 a_6^9 a_5 a_3 + 16204837 a_6^8 a_5^2 a_4 + 2160636 a_6^7 a_5^3 a_4 \\ 
&+ 17082772 a_6^{10} a_2 - 39002537 a_6^8 a_5 a_2 - 4325406 a_6^7 a_5^2 a_3 - 4325406 a_6^6 a_5^3 a_2 - 6488109 a_6^9 a_1 \\ 
&+ 12976218 a_6^7 a_5 a_1 + 4325406 a_6^8 a_0 - 2162703 a_6^6 a_5 a_0 - 2162703 a_1 a_0, \\
e_{16} &= 725416965 a_6^{16} - 2076145350 a_6^{14} a_5 - 250882923 a_6^{13} a_4 + 6393959250 a_6^{12} a_5^2 + 143394719 a_6^{11} a_5 a_4 \\
 &+ 814614675 a_6^{10} a_5^3 + 550745292 a_6^{12} a_3 - 184062332 a_6^{10} a_5 a_3 - 56716807 a_6^9 a_5^2 a_4 - 7562241 a_6^8 a_5^3 a_4 \\ 
 &- 59789752 a_6^{11} a_2 + 136508867 a_6^9 a_5 a_2 + 15138921 a_6^8 a_5^2 a_3 + 15138921 a_6^7 a_5^3 a_2 + 22708381 a_6^{10} a_1 \\
 &- 45416762 a_6^8 a_5 a_1 - 15138921 a_6^9 a_0 + 7569459 a_6^7 a_5 a_0 + 7569459 a_0^2, \\
\end{split}
\]
 \end{tiny}

\end{document}

%% file: macro.tex
\usepackage[T1]{fontenc}
\usepackage[utf8]{inputenc}

\usepackage{amsmath, amssymb, amsthm, amsaddr}

\numberwithin{equation}{section}

\usepackage{graphicx}
\usepackage{xcolor}

\usepackage{booktabs}
\usepackage{enumitem}

\usepackage[msc-links]{amsrefs}

\usepackage{listings}

\usepackage{hyperref}
\hypersetup{
  colorlinks = true,
  urlcolor   = blue,
  linkcolor  = blue,
  citecolor  = blue
}
\usepackage[capitalise]{cleveref} 

\theoremstyle{plain}
\newtheorem{thm}{Theorem}[section]

\newtheorem{lem}[thm]{Lemma}

\theoremstyle{definition}
\newtheorem{defn}[thm]{Definition}
\newtheorem{exa}[thm]{Example}

\theoremstyle{remark}

\crefname{thm}{Thm.}{Thms.}
\crefname{prop}{Prop.}{Props.}
\crefname{lem}{Lem.}{Lems.}
\crefname{cor}{Cor.}{Cors.}
\crefname{conj}{Conj.}{Conjs.}
\crefname{defn}{Def.}{Defs.}
\crefname{exa}{Exa.}{Exas.}
\crefname{rem}{Rem.}{Rems.}

\newcommand{\F}{\mathbb{F}}
\newcommand{\Q}{\mathbb{Q}}
\newcommand{\C}{\mathbb{C}}
\newcommand{\PP}{\mathbb{P}}

\DeclareMathOperator{\Gal}{Gal}
\DeclareMathOperator{\PSL}{PSL}
\DeclareMathOperator{\Stab}{Stab}

\DeclareMathOperator{\res}{Res}

\newcommand{\Res}[3]{\operatorname{Res}_{#1}\!\big(#2,#3\big)}

\lstdefinestyle{sagepy}{
  language=Python,
  basicstyle=\ttfamily\small,
  columns=fullflexible,
  keepspaces=true,
  showstringspaces=false,
  upquote=true,
  frame=single,
  breaklines=true,
  tabsize=2
}
\usepackage{pgfplots}
\usepackage{tikz}
\usepackage{pgfplots}
\pgfplotsset{compat=1.18}